\documentclass{gOMS2e}

\usepackage{epstopdf}
\usepackage{subfigure}
\usepackage{graphicx}
\usepackage{psfrag}
\usepackage{subfigure}
\usepackage{url}
\usepackage{color}
\usepackage{epsfig}
\usepackage{array,amsmath,amssymb}

\theoremstyle{plain}
\newtheorem{theorem}{Theorem}[section]

\newtheorem{lemma}[theorem]{Lemma}

\theoremstyle{definition}
\newtheorem{definition}{Definition}
\theoremstyle{remark}
\newtheorem{remark}{Remark}
\newtheorem{assumption}[theorem]{Assumption}

\newcommand{\rset}{\mathbb{R}}

\newcommand{\Ld}{L_{\mathrm{d}}}

\newcommand{\bo}{}

\newcommand{\lu}{\mathcal{L}}
\newcommand{\ca}{\mathcal{K}}
\providecommand{\norm}[1]{\lVert#1\rVert}

\begin{document}

\articletype{RESEARCH ARTICLE}

\title{Iteration complexity analysis of dual first order methods for
conic convex programming}

\author{
\name{I. Necoara\textsuperscript{a}$^{\ast}$
\thanks{$^\ast$Corresponding author. Email:
ion.necoara@acse.pub.ro.} and A. Patrascu\textsuperscript{a}}
\affil{\textsuperscript{a} Automatic Control and  Systems
Engineering Department, University Politehnica Bucharest, 060042
Bucharest, Romania  \received{January 2015} } }

\maketitle

\begin{abstract}
In this paper we provide a detailed analysis  of the iteration
complexity of dual first order methods for solving conic convex
problems. When it is difficult to project on the primal feasible set
described by convex constraints, we use the Lagrangian relaxation to
handle the complicated constraints and then, we apply dual first
order algorithms for solving the corresponding dual problem. We give
convergence analysis for dual first order algorithms (dual gradient
and fast gradient algorithms): we provide sublinear or linear
estimates on the primal suboptimality and feasibility violation of
the generated approximate primal solutions. Our analysis relies on
the Lipschitz property of the gradient of the dual function or an
error bound property of the dual. Furthermore, the iteration
complexity analysis is based on two types of approximate primal
solutions:  the last primal iterate or an average primal sequence.
\end{abstract}


\begin{classcode}90C25; 90C46; 49N15; 65K05.
\end{classcode}

\section{Introduction}
\label{sec_introduction} Nowadays, many engineering applications can
be posed as  conic  convex problems.  Several
important applications that can be modeled in this framework,  the
network utility maximization  problem
\cite{BecNed:14,KelMau:98,WeiOzd:10}, the resource allocation
problem \cite{XiaBoy:06}, the optimal power flow problem for a power
system \cite{ZimMur:11} or model predictive control problem for a
dynamical system \cite{NecNed:13,NedNec:12,PatBem:12}, have
attracted great attention lately.

\noindent When it is difficult to project on the primal feasible set
of the  convex problem, we use the Lagrangian relaxation to handle
the complicated constraints and then solve the corresponding dual.
First order methods for solving  the corresponding dual of
constrained  convex problems  have been extensively studied in the
literature. Dual subgradient methods based on averaging (so called
ergodic sequence), that produce primal solutions in the limit, can
be found e.g. in  \cite{GusPat:14,KiwLar:07,LarPat:98}. Convergence
rate analysis for the dual subgradient method has been studied e.g.
in \cite{NedOzd:09}, where estimates  for suboptimality and feasibility
violation of an average primal sequence are provided.
In \cite{NecSuy:08} the authors have combined  a
dual fast gradient algorithm and a smoothing technique for solving
non-smooth dual problems and derived  rate of convergence of order
$\mathcal{O}\left(1/k\right)$, with $k$ denoting the
iteration counter,  for primal suboptimality and
feasibility violation for an average primal sequence.   Also, in
\cite{NecNed:13} the authors proposed inexact dual (fast) gradient
algorithms  for solving  dual problems and estimates of order
$\mathcal{O}\left(1/k\right)$ ($\mathcal{O}\left(1/k^2\right)$) in
an average primal sequence  are provided for  primal suboptimality
and feasibility violation. Convergence properties of a  dual fast
gradient algorithm were also analyzed in \cite{PatBem:12} in the
context of predictive control. However, most of the papers
enumerated above provide an approximate primal solution for convex
problems based on averaging.

\noindent There are very few papers deriving the iteration
complexity of dual first order methods using as an approximate
primal solution the last iterate of the algorithm (see e.g.
\cite{LuoTse:93a,NecNed:14a,BecNed:14,BecTeb:14}),
although from our practical experience we have observed that usually these methods are
converging faster in the primal last iterate than in a primal
average sequence. For example, for a dual fast gradient method, rate
of convergence of order $\mathcal{O}\left(1/k\right)$  in the last
iterate is provided in \cite{BecNed:14} under the assumptions of
Lipschitz continuity and strong convexity of the primal objective
function and primal linear constraints. From our knowledge first
result on the linear convergence of dual gradient method in the last
iterate was provided in \cite{LuoTse:93a} under a \textit{local}
error bound property of the dual. However, in \cite{LuoTse:93a}
linear convergence is proved only locally and  for dual gradient
method. Recently, in \cite{NecNed:14a} the authors show that, for
linearly constrained smooth convex problems satisfying a Slater type
condition, the dual problem has a global error
bound property. 
Moreover, in \cite{Tse:10} Tseng posed the question whether there
exist fast gradient schemes  that converge linearly on convex
problems having an error bound property.

\noindent Another strand of this literature uses augmented
Lagrangian based methods  \cite{HonLuo:12,LanMon:08,NedNec:12} or
Newton methods \cite{WeiOzd:10,NecSuy:09}. For example,
\cite{HonLuo:12} established a  linear convergence rate of
alternating direction method of multipliers using an error bound
condition that holds under specific assumptions on the primal
problem. In \cite{LanMon:08,NedNec:12} the iteration complexity of
inexact augmented Lagrangian methods is analyzed, where the inner
problems are solved approximately and the dual variables are updated
using dual (fast) gradient schemes.  In \cite{WeiOzd:10,NecSuy:09}
dual Newton algorithms are derived  under the assumption that the
primal objective function is self-concordant.

\noindent In conclusion, despite the fact that there are attempts to
analyze the convergence properties of dual  first order methods, the
results are dispersed, incomplete and many aspects have not been
fully studied. In particular, in practical applications the main
interest is in finding  approximate primal solutions. Moreover, we
need to characterize the convergence rate for these near-feasible
and near-optimal primal solutions. Finally, we are interested in
providing  schemes with fast convergence rate. These issues motivate
our work here, which provides 
a detailed convergence analyzes of dual first order methods for
solving conic convex~problems.

\vspace{0.04cm}

\noindent \textit{Contributions}. In this paper we  provide a
convergence analysis of  dual first order methods producing
approximate primal feasible and suboptimal solutions for conic
convex  problems. Our analysis is based on the Lipschitz gradient
property of the dual function or an error bound property of the dual
problem. Further, the iteration complexity analysis is based on two
types of approximate primal solutions:  the last primal iterate or
an average primal sequence. We prove that first order algorithms for
solving the dual problem have the following iteration complexity in
terms of primal suboptimality and infeasibility:

\noindent (i) for strongly convex primal objective functions we
prove: for dual gradient method a  sublinear convergence rate in
both, an average primal sequence (convergence rate of order
$\mathcal{O}(1/k)$), or the last primal iterate sequence
(convergence rate $\mathcal{O}(1/\sqrt{k})$)); for dual fast
gradient method a sublinear convergence rate in an average primal
sequence (convergence rate $\mathcal{O}(1/k^2)$), or the last primal
iterate sequence (convergence rate $\mathcal{O}(1/k)$).

\noindent (ii)  if we use regularization techniques we prove that
the convergence estimates  of  dual fast gradient method for both
primal sequences (the last iterate  and an average of iterates) have
the same order (up to a logarithmic factor).

\noindent (iii) if additionally the dual problem has an error bound
property,  then we prove that dual first order methods (including a
 fast gradient scheme with restart) converge \textit{globally} with linear rate
in the last primal iterate sequence (convergence rate
$\mathcal{O}(\theta^k)$, with $\theta<1$), a result which appears to
be new in this area.

\noindent (iv) finally, if the conic constraints are linear
constraints, then based on the properties of dual first order
methods and regularization techniques we  improve the previous
convergence rates of dual first order methods in the last iterate
with one order of magnitude.

\noindent An important feature of our results is that these rates of
convergence are  not only for the average of iterates but also  for
the latest iterate. This feature is of practical importance since
usually the  last iterates are employed in practical applications
and the present paper provides computational complexity certificates
for them.

\vspace{0.1cm}

\noindent \textit{Notations:} We work in the space $\rset^n$
composed by column vectors. For $u,v \in \rset^n$ we denote the
standard Euclidean  inner product $\langle u,v \rangle = u^T v$, the
Euclidean norm $\left \| u \right \|=\sqrt{\langle u,u \rangle}$ and
the projection of $u$ onto the convex set $X$ as $\left[u
\right]_X$. Further, $\text{dist}_X(u)$ denotes the distance from
$u$ to the convex set $X$, i.e. $\text{dist}_X(u) =\min_{x \in X}
\|x-u\|$. Moreover, for a matrix $G \in \rset^{m \times n}$ we use
the notation $\|G\|$ for the spectral norm.


\section{Problem formulation}
\label{sec_formulation}

We consider the following  conic convex  optimization problem:
\begin{align}
\label{eq_prob_princc} f^* = &\min_{u \in U} \; f(u) \quad
\text{s.t.}  \quad  G u + g \in \ca,
\end{align}
where $f: \rset^n \to \rset^{}$ is convex function, $G \in \rset^{p
\times n}$,  $\ca \subseteq \rset^p$ is a proper  cone  and $U
\subseteq \rset^n$ is a closed convex set. Moreover, we assume that
both sets $\ca$  and $U$  are simple, i.e. the projection on these
sets is easy.  Many engineering applications can be posed as
constrained convex  problems \eqref{eq_prob_princc} (e.g. network utility maximization
problem \cite{BecNed:14,XiaBoy:06}: $f$ is $\log$ function and $U$
is a box set;  optimal power flow problem \cite{ZimMur:11}: $f$
is quadratic function and $U$ is a box set;  model predictive
control  problem \cite{NecNed:13,NedNec:12,PatBem:12}: $f$ is quadratic
function and $U$ is a set described by linear equality constraints).
Thus, we are interested in deriving tight convergence  estimates
of dual first order methods for this  optimization model.  We denote with
$\ca^*  \subseteq \rset^p$ the corresponding dual cone of $\ca$,
i.e. $\ca^* =\{x: \; \langle x, u \rangle \geq 0 \;\; \forall u \in
\ca\}$. Further, for simplicity of the exposition we use the short
notation:
\[ g(u) = -Gu - g.   \]

\noindent Throughout the paper, we make the following assumption on
optimization problem \eqref{eq_prob_princc}:
\begin{assumption}
\label{as_strong} The function $f$ is $\sigma_{\mathrm{f}}$-strongly
convex w.r.t. the Euclidean norm and there exists a finite optimal
Lagrange multiplier $x^*$ for the conic constraints of
\eqref{eq_prob_princc}. \qed
\end{assumption}

\noindent Note that if the objective function $f$ is not strongly
convex,  we can apply smoothing techniques by adding a
regularization term to the convex function $f$ in order to obtain a
strongly convex approximation of it and a corresponding smooth
approximation of the dual function. Then, we can use a dual fast
gradient method  for maximizing the smooth approximation of the dual
function and then we can recover an approximate primal solution for
the original problem  (see e.g. \cite{NecSuy:08} for more details
regarding the iteration complexity estimates for this approach).
Furthermore, we can always guarantee the existence of a finite
optimal Lagrange multiplier $x^*$ provided that e.g. Slater
condition holds:  there exists $\tilde u \in \text{relint}(U)$ such
that $G \tilde u + g \in \text{int}(\ca)$.

\noindent Since there exists  a finite optimal Lagrange
multiplier $x^*$,  strong duality holds for optimization problem
\eqref{eq_prob_princc} (see \cite{RocWet:98}). In particular, we
have:
\begin{equation}
\label{eq_dual_prob} f^*=\max_{x \in \ca^*} d(x),
\end{equation}
where $d(x)$ denote the dual function of problem
\eqref{eq_prob_princc}:
\begin{equation}
\label{eq_dual_func} d(x)=\min_{u \in U} \lu (u,x) \quad (= f(u)  +  \langle x, g(u) \rangle).
\end{equation}

\noindent  We  denote by $X^* \subseteq \ca^*$ the set of optimal
solutions of dual problem \eqref{eq_dual_prob}, which is nonempty
and convex according to  Assumption  \ref{as_strong}. Since $f$ is
strongly convex function, the Lagrangian function $\lu(u,x) = f(u) +
\langle x, g(u) \rangle$ is also strongly convex. Then, the inner
problem $\min_{u \in U} \lu (u,x)$ has always a unique finite
optimal solution for any fixed $x \in \rset^p$. In conclusion, the
dual function $d$ has the effective domain the entire Euclidean
space $\rset^p$, i.e. $\text{dom} \ d = \rset^p$. Moreover, since
the minimizer of \eqref{eq_dual_func} for any fixed $x \in \rset^p$
is unique, from Danskin's theorem \cite{RocWet:98} we get that the
dual function $d$ is differentiable everywhere  and its gradient is
given by the following expression:
\begin{equation*}
\nabla d(x) = g(u(x)) \quad (= -G u(x) - g)   \quad \forall x \in
\rset^p,
\end{equation*}
where $u(x)$ denotes the unique optimal solution of the inner
problem \eqref{eq_dual_func}, i.e.:
\begin{equation}
\label{eq_inner_sol} u(x)=\arg\min_{u \in U} \lu(u,x).
\end{equation}
Moreover, from Theorem 1 in \cite{Nes:05} it follows immediately
that the dual gradient $\nabla d$ is Lipschitz continuous on $\ca^*$
with constant $L_\text{d} = \frac{\|G\|^2}{\sigma_\text{f}}$, i.e.:
\begin{align}
\label{lipd} \| \nabla d(x) - \nabla d({y})\| &  \leq
\frac{\|G\|^2}{\sigma_\text{f}} \| x - {y}\| \qquad \forall x, {y}
\in \ca^*.
\end{align}
From Lipschitz continuity of the dual gradient  \eqref{lipd} the
following inequality (so-called descent lemma) is valid
\cite{Nes:04,Nes:05}:
\begin{equation}
\label{eq_descent} d(x) \geq d({y})+ \left\langle \nabla d({y}), x -
{y}\right\rangle -\frac{L_{\text{d}}}{2}\|x - {y}\|^2 \quad \forall
x, {y} \in \ca^*.
\end{equation}

\noindent  In this paper  we analyze several  dual first order
methods for solving problem  \eqref{eq_prob_princc}  and derive  convergence
estimates for  dual and primal suboptimality and also for primal feasibility
violation, i.e. finding an $\epsilon$-primal-dual pair
$\left(\tilde{u},\tilde{x}\right) \in U \times \ca^*$ such that:
\begin{align}
\label{condition_outer} & ~~~~~ \text{dist}_\ca(G \tilde{u} + g)
 \leq \mathcal{O} (\epsilon), \;\;\;  \| \tilde{u}- u^* \| \leq
\mathcal{O} (\epsilon), \\
&-\mathcal{O}(\epsilon)\!\leq \!f(\tilde{u}) -\! f^* \!\!\leq\!
\mathcal{O} (\epsilon) ~~\text{and} ~~f^*-d(\tilde{x}) \leq
\mathcal{O} (\epsilon), \nonumber
\end{align}
where $\epsilon$ is a given accuracy and $u^*$ is the unique minimizer of problem  \eqref{eq_prob_princc}. Thus,  we  introduce the
following definition:
\begin{definition}
We say that $\tilde u \in U$ is an \textit{$\epsilon$-primal solution} for
the original convex problem  \eqref{eq_prob_princc} if we have the
following relations for primal infeasibility and  suboptimality:
\begin{align}
\label{epsilon_primal} & \text{dist}_\ca(G \tilde{u} + g)  \leq
\mathcal{O} (\epsilon) \;\;\;\; \text{and} \;\;\;\;  -\mathcal{O}(\epsilon) \leq f(\tilde{u}) - f^*
\leq  \mathcal{O} (\epsilon).
\end{align}
\end{definition}


\subsection{Preliminary results}
\noindent In this section we derive first some relations between the
optimal solution of the inner problem $u(x)$ and the dual function
$d(x)$. Then, we also derive some properties of the gradient map.
These results  will be used in the subsequent sections.

\noindent In the next lemma we derive some relations between the optimal solution of the
inner problem $u(x)$ and the dual function $d(x)$. These relations
have been proven in \cite{BecNed:14,BecTeb:14} for $\ca = \rset^n_+$
(non-negative orthant). For completeness, we also give a short
proof:
\begin{lemma}
\label{lemma_sc} Under Assumption \ref{as_strong}, the following
inequality holds:
\begin{align}
\label{ineq_x} \frac{\sigma_{\mathrm{f}}}{2}\| \bo{u}(x) - u^*\|^2
\leq f^* - d(x) \qquad \forall x \in \ca^*
\end{align}
and the primal feasibility violation  can be expressed  in terms of
$\| \bo{u}(x) - u^*\|$ as:
\begin{align}
\label{ineq_feas2}
 \text{dist}_\ca(G u(x) + g) \leq \|G\| \ \| \bo{u}(x) - u^* \|
 \qquad \forall x \in \ca^*.
\end{align}
\end{lemma}

\begin{proof} First, let us recall that $g(u(x)) = - G u(x) - g$ and the
following relations:
\[ d(x) = \lu(u(x),x) =f(\bo{u}(x)) + \langle
x, g( \bo{u}(x))  \rangle \;\;  \text{and} \;\; \nabla d(x)=
g(\bo{u}(x)). \] Since $f(u)$ is $\sigma_{\mathrm{f}}$-strongly
convex, it follows that $\mathcal{L}(u,x)$ is also
$\sigma_{\mathrm{f}}$-strongly convex in the variable $\bo{u}$ for
any fixed $x \in \ca^*$, which gives the following inequality:
\begin{align}
\label{eq_scL} \mathcal{L}(u,x) \geq \mathcal{L}(u(x),x) +
\frac{\sigma_{\mathrm{f}}}{2}\|\bo{u}(x) - u\|^2 \quad \forall u \in
U, x \in \ca^*.
\end{align}
Taking now $u = u^* = u(x^*)$ in  the previous inequality \eqref{eq_scL} and
using  that $\nabla d(x^*) = g(u^*)   \in - \ca$ for any $x^* \in X^*$
and that $\langle x, \nabla d(x^*) \rangle \leq 0$ for any $x \in {\ca^*}$,
we have:
\begin{align*}
 \frac{\sigma_{\mathrm{f}}}{2}\|\bo{u}(x) - \bo{u}^*\|^2
&\leq \lu(\bo{u}^*,x) - \lu(u(x),x) = f(u^*)+ \langle x, \nabla
d(x^*)\rangle - d(x) \leq f^*-d(x),
\end{align*}
valid for all $x \in {\ca^*}$.  We now express the primal
feasibility violation  in terms of $\| \bo{u}(x) - u^*\|$ for any $x
\in {\ca^*}$. Indeed, using that $u(x^*) = u^*$ and that $G u^* + g
\in \ca$ we get:
\begin{align*}
\text{dist}_\ca(G u(x) + g) & \leq \| G u(x) + g - (G u^* + g ) \|
 \leq \|G\| \| \bo{u}(x) - u^* \|.
\end{align*}
These relations prove the statements of the lemma.
\end{proof}

\noindent We now express the   primal suboptimality in terms of $\|
\bo{u}(x) - u^*\|$, a result which appears to be new:

\begin{lemma}
Under Assumption \ref{as_strong}, the following inequality holds:
\begin{align}
\label{ineq_opt} |f(u(x)) - f^*| \leq \|G\| \left( \|x - x^*\| +
\|x^*\| \right )   \| \bo{u}(x) -  u^* \| \quad \forall x \in {\ca^*}, \;
x^* \in X^*.
\end{align}
\end{lemma}

\begin{proof} Firstly, using the complementarity condition
$\langle x^*, g(u^*)  \rangle =0$ we get:
\begin{align*}
\langle x^*, g(u^*)  \rangle + f(u^*) & = d(x^*)  = \min_{u \in U} [
f(u) + \langle x^*, g(u) \rangle]  \leq f(u(x)) + \langle x^*,
g(u(x))  \rangle,
\end{align*}
which leads to the following relation:
\[  f(u(x)) - f^* \geq \langle x^*,  g(u^*) - g(u(x))  \rangle. \]
Using the Cauchy-Schwartz inequality we derive:
\begin{align}
\label{ineq_optl} f(u(x)) - f^* & \geq - \|x^*\| \|g(u^*) - g(u(x))
\| \geq  - \|x^*\| \ \|G\| \ \|\bo{u}(x) - u^* \|.
\end{align}
Secondly, from the definition of the dual function we have:
\[ d(x) = f(u(x)) + \langle x,  \nabla d(x) \rangle. \] Subtracting $f^* = d(x^*)$ from both sides and
using the complementarity condition $ \langle x^*, \nabla d(x^*)
\rangle =0 $ we get:
\begin{align}
f(u(x)) - f^* & = d(x) - d(x^*) - \langle
x, \nabla d(x) \rangle \nonumber \\
& = d(x) - d(x^*) - \langle x - x^*, \nabla d(x^*)
\rangle + \langle x, \nabla
d(x^*) - \nabla d(x) \rangle \nonumber\\
& \leq \langle x, \nabla d(x^*) - \nabla d(x) \rangle  \leq \| x \|
\cdot \|g(u^*)  - g(u(x)) \| \nonumber\\
& \leq \| x \|  \|G\| \| \bo{u}(x) - u^* \|, \label{ineq_feas3}
\end{align}
valid for all $x \in {\ca^*}$, where in the first inequality we used
concavity of dual function $d$, in the second inequality the
relation $\nabla d(x) = g(u(x)) = - G u(x) - g$ and Cauchy-Schwartz
inequality and in the third inequality a property of the Euclidean
norm. In conclusion, using the triangle inequality for vector norms,
we obtain the following inequality:
\begin{align}
\label{ineq_optr} & f(u(x)) - f^* \leq  \|G\| \left( \| x - x^* \|
\!+\! \|x^*\| \right)  \| \bo{u}(x) - u^* \| \quad \forall x \in {\ca^*},
\; x^* \in X^*.
\end{align}
Combining \eqref{ineq_optl} and \eqref{ineq_optr} we obtain the
bound on primal suboptimality   \eqref{ineq_opt}.
\end{proof}

\noindent Note that, based on our derivations from above,  we are
able to characterize primal suboptimality \eqref{ineq_opt} without
assuming any Lipschitz property on $f$ as opposed to the results in
\cite{BecNed:14}, where the authors had to require Lipschitz
continuity of $f$ for providing estimates on primal suboptimality.
However, for many applications $U$ is unbounded set and  $f$ is quadratic function
(e.g. in model predictive control $f$ is quadratic and $U$ might be a set described by linear equality constraints \cite{NecNed:13,PatBem:12})  and thus it is not Lipschitz continuous, so that our theory covers this important case.

\noindent Further, let us  introduce the notion of gradient map
denoted  $\nabla^+ d(x)$ and the gradient step from $x$ denoted
$x^+$ (see also \cite{Nes:04}):
\begin{equation}
\label{eq_gradient_map} \nabla^+ d(x)=\left[x + \frac{1}{L_\text{d}}
\nabla d(x)\right]_{\ca^*} - x   \quad \text{and} \quad x^+ =
\left[x + \frac{1}{L_\text{d}} \nabla d(x)\right]_{\ca^*}.
\end{equation}
Clearly, $x^* \in X^*$ if and only if $\nabla^+ d(x^*) =0$ and
$\nabla^+ d(x) =  x^+ - x$. Next lemma proves that the norm of the
gradient map is  decreasing along a gradient step, i.e.:

\begin{lemma}
\label{lemma2_dg} Under Assumption \ref{as_strong} the following
inequality holds:
\begin{align}
\label{decrease_gm} \|\nabla^+ d(x^+)\| \leq \|\nabla^+ d(x)\| \quad
\forall x \in \ca^*.
\end{align}
\end{lemma}

\begin{proof}
Since the dual function $d$ has $L_\text{d}$-Lipschitz gradient on
$\ca^*$ (see \eqref{lipd}) and is concave, the following relation
holds \cite{Nes:04}:
\[  \| \nabla d(y) - \nabla d(x) \|^2 \leq L_\text{d}
\langle \nabla d(y) - \nabla d(x), x-y \rangle \quad \forall x,y \in
\ca^*. \]

\noindent If we replace in the previous inequality $y$ with a
gradient step in $x$, i.e. $y = x^+ = [x + \frac{1}{L_\text{d}}
\nabla d(x)]_{\ca^*}$, and arranging the terms we get:
\[ \| \nabla d(x^+) - \nabla d(x) + \frac{L_\text{d}}{2} (x^+ - x) \|
\leq \frac{L_\text{d}}{2} \|x^+ - x\|.   \] Grouping the terms
appropriately we obtain:
\[ \| (x^+ + \frac{1}{L_\text{d}}\nabla d(x^+)) - (x + \frac{1}{L_\text{d}} \nabla d(x)) + \frac{1}{2} (x - x^+) \|
\leq \frac{1}{2} \|x^+ - x\|. \] Using now the  triangle inequality
for a norm $\|z\| - \|w\| \leq \| z+ w\|$ we get:
\[ \| (x^+ + \frac{1}{L_\text{d}}\nabla d(x^+)) -
(x + \frac{1}{L_\text{d}} \nabla d(x)) \| \leq  \|x^+ - x\|.
\]
Finally, since the projection is  non-expansive  we obtain:
\[ \| [x^+ + \frac{1}{L_\text{d}}\nabla d(x^+)]_{\ca^*} -
[x + \frac{1}{L_\text{d}} \nabla d(x)]_{\ca^*} \| \leq  \|x^+ - x\|.
\]
Combining the previous inequality with  the definitions of $\nabla ^+
d$ and of $x^+$, we obtain the statement of the lemma.
\end{proof}

\noindent Finally we show a  relation between the dual gradient and
the gradient map:
\begin{lemma}
\label{lemma2_dg} Under Assumption \ref{as_strong} the following
inequality holds:
\begin{align}
\label{decrease_ggm} \text{dist}_{\ca} \left(- \nabla d(x) \right) \leq L_\text{d}
\|\nabla^+ d(x)\|  \qquad \forall x \in \ca^*.
\end{align}
\end{lemma}

\begin{proof}
First, we derive a property of the projection, namely:
\begin{align}
\label{prop_proj} [y]_{\ca^*} - y \in \ca \quad \forall y \in
\rset^p.
\end{align}
Indeed, $[y]_{\ca^*} \in \arg \min_{z \in \ca^*} \|z - y\|$ if and
only if $\langle [y]_{\ca^*} -y, z - [y]_{\ca^*} \rangle \geq 0$ for
all $z \in \ca^*$. Hence $\langle [y]_{\ca^*} -y, z \rangle  \geq
\langle [y]_{\ca^*} - y, [y]_{\ca^*} \rangle$ for all $z \in \ca^*$.
Since the left hand side of the last inequality is bounded  below
for all $z \in \ca^*$, it follows that $[y]_{\ca^*} -y \in \ca$.
Then, we have:
\begin{align*}
\|\nabla^+ d(x)\| & = \| x^+ -x \| =  \|[x + \frac{1}{L_\text{d}}
\nabla
d(x)]_{\ca^*} - x\| \\
& = \| \underbrace{[x + \frac{1}{L_\text{d}} \nabla d(x)]_{\ca^*}  -
(x + \frac{1}{L_\text{d}} \nabla d(x))}_{\eqref{prop_proj} \;\;
\Rightarrow \;\; \in \ca} - (-\frac{1}{L_\text{d}}
\nabla d(x)) \| \\
& \geq \text{dist}_{\ca}(-\frac{1}{L_\text{d}} \nabla d(x))) =
\frac{1}{L_\text{d}} \text{dist}_{\ca}(- \nabla d(x))).
\end{align*}
Since $\nabla d(x) =  - G u(x) - g$, we also obtain a bound for
primal infeasibility:
\begin{align}
\label{map_feas} \text{dist}_{\ca}(G u(x) + g) \leq L_\text{d} \|
x^+ - x\|  \qquad \forall x \in \ca^*.
\end{align}
\end{proof}


\subsection{Dual first order algorithms}
\noindent In this section we present  a general framework for  dual first order
methods generating approximate primal feasible and primal optimal
solutions for the convex problem \eqref{eq_prob_princc}. This
general framework covers important particular algorithms \cite{Nes:04,BecTeb:14}:
e.g. dual gradient algorithm,  dual fast gradient algorithm,
hybrid fast gradient/gradient algorithm or restart fast gradient algorithm,
as we will see in the next sections. Thus, we will analyze the iteration complexity of some particular cases   of  the following general dual first order method that updates two dual sequences $(x^k, y^k)$ and one primal sequence $u^k$ as follows:
\begin{center}
\framebox{
\parbox{7.4cm}{
\begin{center}
\textbf{ Algorithm {\bf (DFO)} }
\end{center}
{Given $x^0 = y^1 \in {\ca^*}$, for $k\geq 1$ compute:}
\begin{enumerate}
\item $u^k = \arg \min\limits_{u \in U} \mathcal{L}(u,y^k)$
\item ${x}^{k}=\left[y^k+ \alpha_k \nabla d(y^k)\right]_+$,
\item $y^{k+1} = x^k + \frac{\theta_k -1 }{\theta_{k+1}} (x^k - x^{k-1})$.
\end{enumerate}
}}
\end{center}
where $\alpha_k$ and $\theta_k$ are the parameters  of the method
and in the next sections we show how we can choose them in an
appropriate way. Recall also the following relations: $u^k = u(y^k)$
and ${\nabla} d(y^k)= g(\bo{u}^k)$. Note that if we cannot solve the inner problem $\min_{u \in U} \mathcal{L}(u,y^k)$ (step 1 in Algorithm \textbf{(DFO)}) exactly, but approximatively with some inner accuracy, then our framework allows us to use
approximate solutions $u^k$ and inexact dual gradients. This is
beyond the scope of the present paper, but for more details see e.g.
\cite{NecNed:13,NedNec:12,WanLin:13}.


\section{Rate of convergence of dual  gradient algorithm}
\label{sec_dfom} \noindent In this section we consider a variant of
Algorithm {\bf (DFO)}, where $\theta_k =1$ for all $k \geq 0$. Under
this choice for the parameter $\theta_k$ we have that $y^{k} =
x^{k-1}$ and thus we obtain the following dual gradient algorithm
with variable step size  $\alpha_k$:
\begin{center}
\framebox{
\parbox{6.5cm}{
\begin{center}
\textbf{ Algorithm {\bf (DG)} }
\end{center}
{Given $x^0 \in {\ca^*}$, for $k\geq 0$ compute:}
\begin{enumerate}
\item $u^k = \arg \min\limits_{u \in U} \mathcal{L}(u,x^k)$
\item ${x}^{k+1}=\left[x^k+ \alpha_k \nabla d(x^k)\right]_{\mathcal{K}^*}$,
\end{enumerate}
}}
\end{center}
where $\frac{1}{L_G} \leq \alpha_k \leq \frac{1}{L_{\text{d}}}$ such
that $L_G \geq L_{\text{d}}$ and   recall that ${\nabla} d(x^k)=
g(\bo{u}^k)$. Let us now derive some important properties of the
dual gradient method that will be useful in the following sections.

\begin{lemma}
\label{lemma1_dg} Let Assumption \ref{as_strong} hold and the
sequence  $\left(x^k\right)_{k\geq 0}$ be generated by Algorithm
{\bf (DG)}. Then, the following inequalities are valid:
\begin{align*}
\|x^k - x^*\| & \leq \|x^0 - x^*\|,  \qquad  d(x^{k+1}) \geq d(x^k)
+\frac{L_\text{d}}{2} \| x^{k+1} - x^k \|^2, \\
& \|x^{k+1}\|^2 \leq \|x^k\|^2 + 2 \alpha_k (f^* - f(u^k)) \quad
\forall k \geq 0, x^* \in X^*.
\end{align*}
\end{lemma}

\begin{proof} Based on the update rule for the gradient method we get:
\begin{align*}
& \| x^{k+1} - x \|^2 =  \| x^{k+1} -  x^k + x^k -  x \|^2 \\
& =  \| x^k - x \|^2 + 2 \langle x^{k+1} -
x^k, x^k - x \rangle + \| x^{k+1} - x^k \|^2\\
& = \| x^k - x \|^2 + 2 \langle x^{k+1} - x^k, x^{k+1} - x
\rangle - \| x^{k+1} - x^k \|^2 \\
& \leq  \| x^k \!-\! x \|^2 \!-\! 2 \alpha_k \langle \nabla
d(x^{k}), x \!-\! x^k\rangle \!+\! 2 \alpha_k \Big(  \langle \nabla
d(x^{k}), x^{k+1} \!\!-\! x^k \rangle \!-\! \frac{L_{\text{d}}}{2}
\| x^{k+1} \!-\! x^k \|^2 \Big)\\
& \overset{\eqref{eq_descent}}{\leq} \| x^k \!-\! x \|^2 \!+\!  2
\alpha_k \left(  d(x^{k+1}) - d(x^k) - \langle \nabla d(x^{k}), x
\!-\! x^k\rangle \right ) \qquad  \forall k \geq 0, \; x \in
{\ca^*},
\end{align*}
where the first inequality follows from  the definition of
$x^{k+1}$, i.e. from the property of the projection operator
$\langle x^{k+1} - x^k - \alpha_k  \nabla d(x^k), x - x^{k+1}\rangle
\geq 0$ for any $x \in {\ca^*}$, and $\alpha_k  L_{\text{d}} \leq
1$. In conclusion, for all $k \geq 0$ and  $x \in {\ca^*}$ we
obtain:
\begin{align}
\label{dg_prr} \| x^{k+1} - x \|^2 \leq & \| x^k - x \|^2   + 2
\alpha_k \left( d(x^{k+1}) - d(x^k) - \langle \nabla d(x^k), x - x^k
\rangle \right).
\end{align}
Now, if we take $x = x^* \in X^*$ in \eqref{dg_prr} and use
concavity of $d$, we get that: $$\| x^{k+1} - x^* \|^2 \leq \| x^k -
x^* \|^2 + 2 \alpha_k \left( d(x^{k+1}) - d(x^*) \right) \leq \| x^k
- x^* \|^2.
$$ Thus, we obtain:
\begin{align}
\label{dg_pr1} \| x^k - x^* \| \leq \| x^0 - x^* \| \qquad \forall k
\geq 0, \quad x^* \in X^*.
\end{align}
Moreover, if we take $x = x^k$ in \eqref{dg_prr} and use $\alpha_k
\leq 1/L_\text{d}$, then we get that the dual gradient algorithm is
an ascent method:
\begin{align}
\label{dg_pr2} d(x^{k+1}) \geq d(x^k) + \frac{L_\text{d}}{2} \|
x^{k+1} - x^k \|^2 \qquad \forall k \geq 0.
\end{align}
Finally, if we take $x = 0$ in \eqref{dg_prr}, using that
$d(x^{k+1}) \leq f^*$ and that $f(u^k) = d(x^k) - \langle \nabla
d(x^k),  x^k \rangle $, then we get:
\begin{align}
\label{dg_pr3} \|x^{k+1}\|^2 \leq \|x^k\|^2 + 2 \alpha_k (f^* -
f(u^k)) \quad \forall k \geq 0.
\end{align}
Relations \eqref{dg_pr1}, \eqref{dg_pr2} and \eqref{dg_pr3} prove
the statements of the lemma.
\end{proof}

\noindent Furthermore, for any $x \in {\ca^*}$ we can define the
following finite quantity:
\begin{equation}
\label{rd} \mathcal{R} (x) =  \min\limits_{x^* \in X^*} \|x^* - x\|.
\end{equation}
From  Assumption \ref{as_strong} it follows that  there exists a
finite optimal Lagrange multiplier $x^*$ and thus $\mathcal{R} (x) <
\infty$, i.e. it  is finite, for any finite $x \in {\ca^*}$. The
well-known sublinear convergence rate  of Algorithm {\bf (DG)} in
terms of dual suboptimality is given in the next lemma (see Theorem
4 in \cite{Nes:12_comp}):

\begin{lemma} \cite{Nes:12_comp}
\label{th_sublin} Let Assumption \ref{as_strong} hold and the
sequence $\left(x^k\right)_{k\geq 0}$ be generated by Algorithm {\bf
(DG)}. Then,  defining $\mathcal{R}_\text{d}  = \mathcal{R}(x^0)$, a sublinear
estimate on dual suboptimality for  dual problem
\eqref{eq_dual_prob} is given by:
\begin{equation}
\label{bound_dual_optim_adg} f^* - d({x}^{k}) \leq \frac{4 L_G
\mathcal{R}_\text{d}^2}{k}.
\end{equation}
\end{lemma}
\begin{proof}
Although the convergence rate is  given for constant step size in
\cite{Nes:12_comp}, it is easy  to show that  for variable step size
the convergence rate is similar. Therefore, we omit the proof and we
refer e.g. to  Theorem 4 in \cite{Nes:12_comp} for details.
\end{proof}

\noindent In the sequel  we use $x^* = [x^0]_{X^*}$ and thus
$\mathcal{R}_\text{d} = \| x^0 - x^*\|$. Our iteration complexity
analysis for Algorithm {\bf (DG)} is based on two types of
approximate primal solutions:  the last primal iterate sequence
$(u^k)_{k \geq 0} $ or an average primal sequence $(\hat u^k)_{k
\geq 0}$ of the form:
\begin{align}
\label{av_dg} \hat u^k =\frac{ \sum_{j=0}^k \alpha_j u^j}{S_k},
\qquad \text{with} \quad S_k=\sum_{j=0}^k \alpha_j.
\end{align}


\subsection{Sublinear convergence in the last primal iterate}
\label{sublinear_first} \noindent In this section we derive
sublinear estimates for  primal feasibility  and primal
suboptimality for the last primal iterate sequence $(u^k)_{k \geq
0}$ generated by  Algorithm \textbf{(DG)}. Let us notice that from
the definition of Algorithm \textbf{(DG)} we have $u^k = u(x^k)$.

\begin{theorem}
\label{th_dglast} Let Assumption \ref{as_strong} hold and the
sequences $\left(x^k,u^k\right)_{k\geq 0}$ be generated by Algorithm
{\bf (DG)}. Then, for a given accuracy $\epsilon>0$ we get an
$\epsilon$-primal solution for \eqref{eq_prob_princc} in the last
primal  iterate $u^k$ of Algorithm \textbf{(DG)} after $k =
{\mathcal O} (\frac{1}{\epsilon^2})$ iterations.
\end{theorem}
\begin{proof}  Firstly, combining \eqref{ineq_x} and
\eqref{bound_dual_optim_adg} we obtain the following important
relation characterizing the distance from the last iterate $u^k$ to
the unique optimal solution $u^*$ of our original problem
\eqref{eq_prob_princc}:
\begin{align}
\label{dist_sublin_dg} \| u^k - u^* \| \leq  \sqrt{\frac{8 L_G
\mathcal{R}_\text{d}^2}{k \sigma_\text{f}}}.
\end{align}
Secondly, combining the previous relation \eqref{dist_sublin_dg} and
\eqref{ineq_feas2} we obtain a sublinear estimate for feasibility
violation of the last iterate $u^k$ for Algorithm \textbf{(DG)}:
\begin{align}
\label{fes_sublin}  \text{dist}_{\mathcal{K}}(Gu^k + g)  & \leq
\|G\| \| \bo{u}^k - u^* \| \ \leq  3 \|G\| \sqrt{\frac{ L_G
 \mathcal{R}_\text{d}^2}{k \sigma_\text{f}}}  = 3\sqrt{\frac{\|G\|^2}{\sigma_\text{f}} \frac{L_G
\mathcal{R}_\text{d}^2}{k}} \leq \frac{3 L_G
\mathcal{R}_\text{d}}{\sqrt{k}},
\end{align}
where we used that $L_\text{d} = \|G\|^2/\sigma_\text{f}$ and
$L_\text{d} \leq L_G$. Finally, we derive a sublinear estimate for
primal suboptimality of the last iterate $u^k$. Combining
\eqref{dg_pr1}, \eqref{ineq_opt} and \eqref{dist_sublin_dg}  we
obtain:
\begin{align}
\label{opt_sublin} |f(u^k) - f^*| & \leq   3 \|G\| (\|x^k - x^*\| +
\|x^*\|)  \sqrt{\frac{ L_G \mathcal{R}_\text{d}^2}{k \sigma_\text{f}}}
\leq  3 \|G\|(2 \|x^0 - x^*\| + \|x^0\|) \sqrt{\frac{ L_G \mathcal{R}_\text{d}^2}{k \sigma_\text{f}}} \nonumber \\
& \leq  (6  \mathcal{R}_\text{d} + 3 \|x^0\|) \frac{L_G
\mathcal{R}_\text{d}}{\sqrt{k}},
\end{align}
where in the second inequality we used   the definition of the
finite constants $\mathcal{R}_\text{d} = \| x^0 - x^*\|$ and
$L_\text{d} = \|G\|^2/\sigma_\text{f}$. In conclusion, we have
obtained sublinear estimates  of order
$\mathcal{O}(\frac{1}{\sqrt{k}})$ for primal infeasibility
(inequality \eqref{fes_sublin}) and primal suboptimality (inequality
\eqref{opt_sublin}) for the last primal iterate sequence $(u^k)_{k
\geq 0}$ generated by  Algorithm~\textbf{(DG)}. Now, it is
straightforward to see that if we want to get an $\epsilon$-primal
solution in $u^k$ we need to perform $k = {\mathcal O}
(\frac{1}{\epsilon^2})$ iterations.
\end{proof}


\subsection{Sublinear convergence in an average  primal sequence}
\label{sublinear_av} \noindent In this section we derive sublinear
convergence  estimates for  primal infeasibility  and primal
suboptimality for the average primal sequence $({\hat u}^k)_{k \geq
0}$ defined  in \eqref{av_dg}.

\begin{theorem}
\label{th_dgav} Let Assumption \ref{as_strong} hold and the
sequences $\left(x^k,u^k\right)_{k\geq 0}$ be generated by Algorithm
{\bf (DG)}. Then, for a given accuracy $\epsilon>0$ we get an
$\epsilon$-primal solution for \eqref{eq_prob_princc} in the average
primal sequence  $\hat u^k$ of Algorithm \textbf{(DG)} after $k =
{\mathcal O} (\frac{1}{\epsilon})$ iterations.
\end{theorem}

\begin{proof} Our proof follows similar lines as in
\cite[Proposition 1]{NedOzd:09} given for the dual subgradient
method. However, in our case, by taking into account that the dual
is smooth and the nice properties of gradient method (see Lemma
\ref{lemma1_dg}) and of the projection on cones \eqref{prop_proj},
we get better convergence estimates than in \cite{NedOzd:09}. First,
given the definition of $x^{j+1}$ in Algorithm \textbf{(DG)} we get:
\[  \left[x^j + \alpha_j \nabla d(x^j)\right]_{\mathcal{K}^*} = x^{j+1} \quad \forall j \geq 0.\]
Subtracting $x^j$ from both sides, adding up the above inequality
for $j\!=\!0$ to $j\!=\!k$, we get:
\begin{align*}
\left\|\sum_{j=0}^k \left[x^j + \alpha_j \nabla
d(x^j)\right]_{\mathcal{K}^*} - x^j\right\| = \norm{x^{k+1} - x^0}.
\end{align*}
Denoting $z^j \!=\! \frac{1}{\alpha_j} \left( \left[ x^j \!+\!
\alpha_j \nabla d(x^j) \right]_{\mathcal{K}^*} \!\!-  (x^j \!+\!
\alpha_j \nabla d(x^j)) \right)
\!\!\overset{\eqref{prop_proj}}{\in}\!\! \ca$ and dividing by $S_k$
we get:
\begin{align*}
\left\| \left(\frac{1}{S_k}\sum_{j=0}^k \alpha_j z^j \right) +
\frac{1}{S_k}\sum_{j=0}^k \alpha_j \nabla d(x^j)  \right\| =
\frac{1}{S_k}\norm{x^{k+1} - x^0}.
\end{align*}
Since $z^j \in \mathcal{K}$,  then also $
\frac{1}{S_k}\sum\limits_{j=0}^k \alpha_j z^j \in \mathcal{K}$.
Moreover, from the definition of $\hat u^k$ and the relation $\nabla
d (x^j) = - Gu^j - g$, we obtain $\frac{1}{S_k}\sum_{j=0}^k \alpha_j
\nabla d(x^j) = - G \hat u^k - g$. Using the definition of the
distance and the previous facts  we obtain:
\begin{equation}\label{subl_1}
 d_{\mathcal{K}}(G\hat{u}^k+g) \le
\left\| \left(\frac{1}{S_k}\sum_{j=0}^k \alpha_j z^j \right) - (G
\hat u^k + g) \right\| = \frac{\norm{x^{k+1} - x^0}}{S_k}.
\end{equation}

\noindent It remains to bound $\norm{x^{k+1} -x^0}$. Using the
inequality \eqref{dg_pr1} we get:
\begin{align*}
\norm{x^{k+1} - x^0}  \leq (\|x^{k+1} - x^*\| + \|x^* -x^0\|)
\overset{\eqref{dg_pr1}}{\leq} 2 \| x^* - x^0\|  = 2
\mathcal{R}_\text{d}.
\end{align*}
Using this bound in  \eqref{subl_1} and the fact that  $S_k =
\sum_{j=0}^k \alpha_j \geq \frac{k+1}{L_G}$, we get the following
estimate on feasibility violation:
\begin{align}
\label{subl_2} d_{\mathcal{K}}(G\hat{u}^k+g)  \leq  \frac{2
\mathcal{R}_\text{d}}{S_k} \leq  \frac{2 L_G
\mathcal{R}_\text{d}}{k+1}.
\end{align}

\noindent In order to derive estimates for primal suboptimality we
use the definition of dual cone $\ca^*$ and that $x^* \in \ca^*$,
which imply:
\begin{align*}
f^* &=  \min_{u \in U} f(u) + \langle x^*, g(u) \rangle \leq f(\hat
u^k) + \langle x^*, g(\hat u^k) \rangle \\
& =  f(\hat u^k) + \langle x^*, -G \hat u^k - g \rangle \leq f(\hat
u^k) + \langle x^*,  [G \hat u^k + g]_{\ca} - (G \hat u^k + g)
\rangle \\
& \leq f(\hat u^k) + \norm{x^*} \text{dist}_{\mathcal{K}}(G\hat{u}^k
+ g)  \overset{\eqref{subl_2}}{\leq} f(\hat u^k) +  \frac{2 L_G
\mathcal{R}_\text{d}}{k+1}  \|x^*\|.
\end{align*}
Using the definition of $\mathcal{R}_\text{d}$ and the previous
inequality  we get:
\begin{align}
\label{subl_3}  f(\hat u^k) - f^* \geq -  \frac{2 L_G
\mathcal{R}_\text{d}}{k+1}(\mathcal{R}_\text{d} + \|x^0\|).
\end{align}
On the other hand, from \eqref{dg_pr3} we have:
\begin{align*}
\|x^{j+1}\|^2 \leq \|x^j\|^2 + 2 \alpha_j (f^* - f(u^j)) \quad
\forall j \geq 0.
\end{align*}
Adding up these inequalities for $j=0$ to $j=k$ we obtain:
\[ \|x^{k+1}\|^2 + 2S_k \sum_{j=0}^k \frac{\alpha_j}{S_k}(f(u^j) - f^*) \leq  \|x^0\|^2.   \]
Using the definition of $\hat u^k$ and the convexity of $f$ we get:
\begin{align}
\label{dg_pr44} f(\hat u^k) - f^*
  \leq   \frac{\|x^0\|^2}{2S_k}.
\end{align}
Combining \eqref{subl_3} and  \eqref{dg_pr44} we  obtain the
following  bounds on primal suboptimality:
\begin{align}
\label{dg_pr4} -  \frac{2 L_G \mathcal{R}_\text{d}}{k+1}
(\mathcal{R}_\text{d} + \|x^0\|) \leq f(\hat u^k) - f^* \leq
\frac{L_G \|x^0\|^2}{2(k+1)}.
\end{align}
In conclusion, we have obtained sublinear estimates  of order
$\mathcal{O}(\frac{1}{k})$ for primal infeasibility (inequality
\eqref{subl_2}) and primal suboptimality (inequality \eqref{dg_pr4})
for the average primal  sequence $(\hat u^k)_{k \geq 0}$ generated
by Algorithm~\textbf{(DG)}. Now, if we want to get an
$\epsilon$-primal solution in $\hat u^k$ we need to perform $k =
{\mathcal O} (\frac{1}{\epsilon})$ iterations.
\end{proof}


\noindent We can also characterize the distance from $\hat u^k$ to
the unique primal solution $u^*$ of problem \eqref{eq_prob_princc}.
Indeed, taking $x = x^*$ and $u = \hat u^k$ in \eqref{eq_scL} and
using that $u(x^*) = u^*$, we have:
\begin{align*}
 \frac{\sigma_{\mathrm{f}}}{2} \| \hat u^k - \bo{u}^*\|^2
& \leq \lu(\hat u^k,x^*) - \lu(u^*,x^*) =  \lu(\hat u^k,x^*) - f^*\\
& =  f(\hat u^k) + \langle x^*, g(\hat u^k) \rangle  - f^* \\
& \leq f(\hat u^k) - f^* + \norm{x^*} \text{dist}_{\mathcal{K}}
(G\hat u^k + g).
\end{align*}
Therefore, we obtain:
\begin{align}
\label{distance} \| \hat u^k - \bo{u}^*\|^2 \leq
\frac{2}{\sigma_{\mathrm{f}}} \left[  f(\hat u^k) - f^*  +
\norm{x^*} \text{dist}_{\mathcal{K}} (G\hat u^k + g) \right].
\end{align}
Using now  \eqref{subl_2}  and \eqref{dg_pr4}, we get:
\[   \| \hat u^k - \bo{u}^*\|^2  \leq
\frac{1}{k+1} \left[   \frac{L_G \|x^0\|^2}{\sigma_{\mathrm{f}}}  +
 \frac{4 L_G \mathcal{R}_\text{d}}{\sigma_{\mathrm{f}}}
(\mathcal{R}_\text{d} + \|x^0\|) \right]. \]

\noindent Note that if we assume constant step size $\alpha_k
=1/L_\text{d}$, then $L_G= L_\text{d}$.  Thus, this choice for the
step size provides us the  best convergence estimates. Further, the
iteration complexity estimates of order
$\mathcal{O}(\frac{1}{\sqrt{k}})$ in the last primal  iterate
sequence $u^k$ (see Section \ref{sublinear_first}) are inferior to
those estimates of order $\mathcal{O}(\frac{1}{k})$ corresponding to
an average of primal iterates $\hat u^k$ (see this section).
However, in Section \ref{extensions} we show that for the particular
case of linearly constrained convex problems the convergence
estimates for both sequences, the last iterate and an average of
iterates, have the same order.


\section{Rate of convergence of dual  fast gradient algorithm}
\label{sec_dfg} \noindent In this section we consider a variant of
Algorithm {\bf (DFO)}, where the step size  $\alpha_k$ is chosen
constant, i.e. $\alpha_k = \frac{1}{L_\text{d}}$ for all $k \geq 0$
and $\theta_k$ is updated iteratively as shown below. In this case
we obtain the following dual fast gradient algorithm, which is an
extension of Nesterov's optimal gradient method \cite{Nes:04} (see
 \cite{BecTeb:14,Tse:08,Tse:10}):
\begin{center}
\framebox{
\parbox{7cm}{
\begin{center}
\textbf{ Algorithm {\bf (DFG)} }
\end{center}
{Given $x^0 = y^1 \in {\ca^*}$, for $k\geq 1$ compute:}
\begin{enumerate}
\item $u^k = \arg \min\limits_{u \in U} \mathcal{L}(u,y^k)$
\item ${x}^{k}=\left[y^k+ \frac{1}{L_\text{d}} \nabla d(y^k)\right]_{\mathcal{K}^*}$,
\item $y^{k+1} = x^k + \frac{\theta_k -1 }{\theta_{k+1}} (x^k - x^{k-1})$,
\item[] $\text{where} \quad  \theta_{k+1} = \frac{1 + \sqrt{1 + 4
\theta_k^2 }}{2}$ and $\theta_1=1$.
\end{enumerate}
}}
\end{center}
where we  recall that $u^k = u(y^k)$   and  ${\nabla} d(y^k)=
g(\bo{u}^k)$. Since $f$ is strongly convex, the Lagrangian $\lu(u,y)
= f(u) - \langle Gu + g, y \rangle$ is also strongly convex for any
fixed $y \in \rset^p$.  Therefore,  the minimizer of \eqref{eq_dual_func} for any fixed $y \in \rset^p$ is unique and from Danskin's theorem \cite{RocWet:98} we get that the
dual function $d$ is differentiable everywhere  and thus $\nabla d(y^k)$ is well defined even for $y^k \not \in \ca^*$. It can be easily seen that $\theta_0= 0$ and the
step size sequence $\theta_k$ satisfies: $\theta_k + \frac{1}{2} \le
\theta_{k+1} \le \theta_k + 1$. Therefore, we obtain the following
bound:
\begin{equation}\label{theta_1}
\frac{k+1}{2} \le \theta_k \le k.
\end{equation}
Rearranging the terms in the step size update $\theta_{k+1}$, we
have the relation:  $\theta_{k+1}^2 - \theta_k^2 = \theta_{k+1}$.
Summing on the history and defining  $S_k^\theta =
\sum\limits_{j=0}^k \theta_j$, we also obtain:
\begin{equation}\label{theta_2}
S_k^\theta = \theta_k^2.
\end{equation}

\noindent Denoting $w^k = x^{k-1}+ \theta_k(x^k - x^{k-1})$  and
$\Delta(x,y) = d(y) + \langle \nabla d(y), x- y\rangle -d(x)$, we
now state the following auxiliary result (for a similar result
corresponding to another formulation of Algorithm {\bf (DFG)} see
\cite{Tse:08}).

\begin{theorem}
\label{corr2}  Let Assumption \ref{as_strong} hold and the sequences
$(x^k,y^k)_{k \ge 0}$ be generated by Algorithm {\bf (DFG)}, then
for any Lagrange multiplier $x \in \mathcal{K}^*$ and $k \geq 0$ we
have the following relation:
\begin{equation}\label{corr2_relation}
\theta_{k+1}^2 (d(x) - d(x^{k+1})) +
\sum\limits_{i=1}^{k+1}\theta_{i}\Delta(x,y^i) +
\frac{L_\text{d}}{2}\norm{w^{k+1}-x}^2 \le
\frac{L_\text{d}}{2}\norm{x^0-x}^2.
\end{equation}
\end{theorem}

\begin{proof} From the Lipschitz gradient relation and the strong convexity
property of the corresponding quadratic approximation of
\eqref{eq_descent}, we have:
\begin{align*}
d(x^{k+1}) &\ge d(y^{k+1}) + \langle \nabla d(y^{k+1}), x^{k+1}-y^{k+1} \rangle - \frac{L_\text{d}}{2}\norm{x^{k+1}-y^{k+1}}^2\\
& \ge d(y^{k+1}) + \langle \nabla d(y^{k+1}), y-y^{k+1}\rangle -
\frac{L_\text{d}}{2}\norm{y-y^{k+1}}^2
 + \frac{L_\text{d}}{2}\norm{y-x^{k+1}}^2  \quad \forall y \in \ca^*.
\end{align*}
 Taking now $x \in \ca^*$, then $\tilde{y} =
\left(1-\frac{1}{\theta_{k+1}}\right)x^k + \frac{1}{\theta_{k+1}}x
\in \ca^*$ and we have:
\begin{align*}
 d(x^{k+1}) & \ge d(y^{k+1}) + \langle \nabla d(y^{k+1}), \tilde{y}-y^{k+1}\rangle - \frac{L_\text{d}}{2}\norm{\tilde{y} - y^{k+1}}^2 + \frac{L_\text{d}}{2}\norm{\tilde{y} - x^{k+1}}^2\\
& = d(y^{k+1}) + \left(1-\frac{1}{\theta_{k+1}}\right)\langle \nabla d(y^{k+1}), x^k - y^{k+1}\rangle \\
& \qquad \!+\! \frac{1}{\theta_{k+1}}\langle \nabla d(y^{k+1}\!), x -
y^{k+1}\rangle \!-
\frac{L_\text{d}}{2\theta_{k+1}^2}\norm{w^k \!- x}^2 \!+ \frac{L_\text{d}}{2\theta_{k+1}^2}\norm{w^{k+1} \!- x}^2 \\
& \ge \left(1- \frac{1}{\theta_{k+1}}\right) d(x^k) +
\frac{1}{\theta_{k+1}}(d(x) + \Delta(x,y^{k+1}))
 - \frac{L_\text{d}}{2\theta_{k+1}^2}\norm{w^k - x}^2 \\
& \hspace{7cm} + \frac{L_\text{d}}{2\theta_{k+1}^2}\norm{w^{k+1} -
x}^2,
\end{align*}
where in the last inequality we used concavity of $d$. Subtracting
now $d(x)$ and multiplying with $\theta_{k+1}^2$ both hand sides, we
obtain:
\begin{align*}
& \theta_{k+1}^2(d(x^{k+1}) - d(x))\\
& \ge \theta_{k+1}(\theta_{k+1} \!-\! 1)(d(x^k) - d(x))\! +\!
\theta_{k+1}\Delta(x, y^{k+1})
\!-\! \frac{L_\text{d}}{2}\norm{w^k \!\!- x}^2 \!+\! \frac{L_\text{d}}{2}\norm{w^{k+1} \!\!- x}^2\\
& =\theta_k^2 (d(x^k) - d(x)) + \theta_{k+1}\Delta(x, y^{k+1}) -
\frac{L_\text{d}}{2}\norm{w^k-x}^2 +
\frac{L_\text{d}}{2}\norm{w^{k+1}-x}^2.
\end{align*}
Further, note that the choice $\theta_1=1$ in Algorithm {\bf (DFG)}
implies that $\theta_0=0$. On the other hand, using the iteration of
Algorithm {\bf (DFG)} we have $\norm{w^0 - x} = \norm{y^1 -
\left(1-\frac{1}{\theta_1}\right)x^0 - \frac{1}{\theta_1}x} =
\norm{x^0 -x}$. Then, summing on the history, we obtain our result.
\end{proof}


\noindent The sublinear convergence rate of Algorithm {\bf (DFG)} in
terms of dual suboptimality is given in the next lemma.
\begin{lemma} \cite{BecTeb:14,Tse:08}
\label{th_sublin_dfg} Let Assumption \ref{as_strong} hold and the
sequences $\left(x^k, y^k\right)_{k\geq 0}$ be generated by
Algorithm {\bf (DFG)}. Then, a sublinear  estimate on dual
suboptimality for dual problem \eqref{eq_dual_prob} is given by
(recall that $\mathcal{R}_\text{d}   =
\min\limits_{x^* \in X^*} \| x^0 - x^*\|$):
\begin{equation}
\label{bound_dual_optim_dfg} f^* - d({x}^{k}) \leq \frac{2
L_\text{d} \mathcal{R}_\text{d}^2}{(k+1)^2}.
\end{equation}
\end{lemma}

\noindent Our iteration complexity analysis for Algorithm {\bf
(DFG)}  is based on two types of approximate primal solutions:  the
last primal iterate sequence $(v^k)_{k \geq 0}$ defined as
\begin{align}
\label{av_dfg1} v^k = u(x^k) =\arg \min\limits_{v \in U}
\mathcal{L}(v,x^k),
\end{align}
or an average primal sequence $(\hat u^k)_{k \geq 0}$ of the form
\begin{align}
\label{av_dfg2} \hat u^k =\sum_{j=0}^k \frac{\theta_j}{S_k^\theta}
u^j,   \quad \text{with} \quad S_k^\theta = \sum\limits_{j=0}^k
\theta_j.
\end{align}


\subsection{Sublinear convergence in the last primal iterate}
\label{sec_dfglast} \noindent In this section we derive sublinear
convergence estimates for  primal infeasibility  and suboptimality
for the last primal iterate sequence $(v^k)_{k \geq 0}$ as defined
in \eqref{av_dfg1} of Algorithm~\textbf{(DFG)}.

\begin{theorem}
\label{th_dfglast}  Let Assumption \ref{as_strong} hold and the
sequences $\left(x^k, y^k, u^k\right)_{k\geq 0}$ be generated by
Algorithm {\bf (DFG)}.  Then, for a given accuracy $\epsilon>0$ we
get an $\epsilon$-primal solution for \eqref{eq_prob_princc} in the
last primal iterate  $v^k = u(x^k)$ of Algorithm \textbf{(DFG)}
after $k = {\mathcal O} (\frac{1}{\epsilon})$ iterations.
\end{theorem}

\begin{proof} Let us notice that $v^k = u(x^k)$ (see \eqref{av_dfg1}). Firstly,
combining \eqref{ineq_x} and \eqref{bound_dual_optim_dfg} we obtain
the following important relation characterizing the distance from
the last iterate $v^k$ to the unique optimal solution $u^*$ of our
original problem \eqref{eq_prob_princc}:
\begin{align}
\label{dist_sublin_dfg} \| v^k - u^* \| \leq  \sqrt{\frac{
L_\text{d}}{\sigma_\text{f}}} \frac{2 \mathcal{R}_\text{d}}{k+1}.
\end{align}
Secondly, combining the previous relation \eqref{dist_sublin_dfg}
and \eqref{ineq_feas2} we obtain a sublinear estimate for
feasibility violation of the last iterate $v^k$ for Algorithm
\textbf{(DFG)}:
\begin{align}
\label{fes_sublin_f} \text{dist}_{\mathcal{K}}(Gv^k + g)  & \leq
\|G\| \| \bo{v}^k - u^* \| \leq  \|G\|
\sqrt{\frac{L_\text{d}}{\sigma_\text{f}}}
 \frac{2\mathcal{R}_\text{d}}{k+1} \nonumber \\
& = \sqrt{\frac{L_\text{d} \|G\|^2}{\sigma_\text{f}}}
\frac{2\mathcal{R}_\text{d}}{k+1}  = \frac{2 L_\text{d}
\mathcal{R}_\text{d}}{k+1},
\end{align}
where we again used $L_\text{d} = \|G\|^2/\sigma_\text{f}$. Finally,
we derive a sublinear estimate for primal suboptimality of the last
iterate $v^k$. We first prove  that $\norm{x^k - x^*} \le
\norm{x^0-x^*}$.  Indeed, taking $x = x^*$ in Theorem \ref{corr2}
and using that  the terms $\theta_{k+1}(f^*-d(x^{k+1}))$ and
$\sum\limits_{i=1}^{k+1}\theta_{i} \Delta(x^*,y^i)$ are positive we
have:
\begin{align*}
\norm{x^0-x^*} & \geq   \norm{w^{k+1}- x^*} =
\theta_{k+1}\norm{x^{k+1} - x^* - \left(1 -
\frac{1}{\theta_{k+1}}\right) (x^k  - x^*) }.
\end{align*}
Using the triangle inequality and dividing by
$\theta_{k+1}$, we further have:
\begin{align*}
\norm{x^{k+1} - x^*} &\le \left(1- \frac{1}{\theta_{k+1}}\right) \norm{x^k - x^*}
+ \frac{1}{\theta_{k+1}}\norm{x^0-x^*} \\
& \le \max \{ \norm{x^k  -x^*}, \norm{x^0 - x^*} \}.
\end{align*}
Using an inductive argument, we can conclude that:
\begin{align}
\label{dfg_pr1} \norm{x^k - x^*} \le \norm{x^0-x^*} \qquad
\forall k \ge 0.
\end{align}
Combining \eqref{dfg_pr1} with relations \eqref{ineq_opt} and
\eqref{dist_sublin_dfg} and using the definition of
$\mathcal{R}_\text{d}$ we obtain:
\begin{align}
\label{opt_sublin_f} |f(v^k)-f^*|  & = |f(u(x^k))-f^*| \le
\left(\norm{x^0-x^*} + \norm{x^*} \right) \|G\| \sqrt{\frac{
L_\text{d}}{\sigma_\text{f}}}
\frac{2\mathcal{R}_\text{d}}{k+1} \nonumber \\
& \leq (2\mathcal{R}_\text{d} + \| x^0\|)
\frac{2L_\text{d}\mathcal{R}_{\text{d}}}{k+1}.
\end{align}
\noindent In conclusion, we   have obtained sublinear estimates  of
order $\mathcal{O}(\frac{1}{k})$ for primal infeasibility
(inequality \eqref{fes_sublin_f}) and primal suboptimality
(inequality \eqref{opt_sublin_f}) for the last primal  iterate
sequence $(v^k)_{k \geq 0}$ generated by Algorithm \textbf{(DFG)}.
Now, if we want to get an $\epsilon$-primal solution in $v^k$ we
need to perform $k = {\mathcal O} (\frac{1}{\epsilon})$ iterations.
\end{proof}

\noindent In \cite{BecNed:14} estimates  of order
$\mathcal{O}(\frac{1}{k})$ have been given  for primal infeasibility
and suboptimality for the last primal  iterate $v^k$ generated by
Algorithm \textbf{(DFG)}.  However, those derivations are based on
the assumption of Lipschitz continuity of the objective function
$f$, while in our derivations we do not need to impose this
additional condition,  since our proofs make use explicitly of the
properties of the algorithm as given in Theorem \ref{corr2} and the
inequality \eqref{dfg_pr1}. Note that for some applications the
assumption of Lipschitz continuity of  objective function $f$ may be
conservative: e.g.  quadratic objective function $f$ and unbounded
set $U$.

\vspace{0.2cm}

\noindent Finally, we consider the application of dual fast
gradient Algorithm (\textbf{DFG}) for the regularization of the dual
problem of \eqref{eq_prob_princc}, i.e.:
\begin{align} \label{dual_reg_cone}
d_\delta^* = \max_{x \in \ca^*} \; d_\delta(x)  \qquad \left( = d(x)
- \frac{\delta}{2} \|x - x^0\|^2 \right).
\end{align}
Note that regularization strategies have been also used  in other
papers, e.g. in order to make the norm of the gradient of some
objective function small by using first order methods
\cite{Nes:12,LanMon:08}. We show in the sequel that by
regularization  we can improve substantially the convergence rate of
dual fast gradient method  in the last iterate.   Denoting
$x^*_\delta$ the optimal solution of \eqref{dual_reg_cone}, its
optimality conditions are given by:
\begin{equation}\label{optcond_regdelta}
 \langle Gu(x^*_{\delta}) + g + \delta (x^*_{\delta}  - x^0), x  - x^*_{\delta}\rangle
 \ge 0 \qquad \forall x \in \mathcal{K}^*.
\end{equation}
Note that the regularized dual objective function $d_\delta (\cdot)$
in \eqref{dual_reg_cone} is strongly concave with
$\sigma_{\text{d},\delta}=\delta$ and has Lipschitz gradient with
$L_{\text{d},\delta} = L_\text{d} + \delta$. Then, if we replace in
Step 3 of Algorithm (\textbf{DFG}) the term $\frac{\theta_k -
1}{\theta_{k+1}}$ with the constant term $
\frac{\sqrt{L_{\text{d},\delta}} -
\sqrt{\sigma_{\text{d},\delta}}}{\sqrt{L_{\text{d},\delta}} +
\sqrt{\sigma_{\text{d},\delta}}}$, i.e.:
\[  y^{k+1} = x^k  + \frac{\sqrt{L_{\text{d},\delta}} -
\sqrt{\sigma_{\text{d},\delta}}}{\sqrt{L_{\text{d},\delta}} +
\sqrt{\sigma_{\text{d},\delta}}} \left( x^k - x^{k-1} \right),  \]
the modified dual fast gradient algorithm achieves linear
convergence \cite{Nes:04}. More precisely,  for solving  the
regularized dual problem \eqref{dual_reg_cone}  with the modified
Algorithm (\textbf{DFG}) described above we have the convergence
rate:
\[  d_\delta^* - d_\delta(x^k) \leq  \left( 1 - \sqrt{\frac{\delta}{L_\text{d} +
\delta}} \right)^k \frac{L_\text{d} + 2 \delta}{2} \| x^0 -
x^*_\delta \|^2.  \]

\noindent We want to find first an upper bound on $\| x^0 -
x^*_\delta \|$ in terms of  $\mathcal{R}_{\text{d}}$. Since
$d_\delta (\cdot)$ is $\delta$-strongly concave function and
$d(x^*_\delta) \leq d(x^*)$, we have:
\begin{align*}
\|x^* - x^*_\delta \|^2 & \leq \frac{2}{\delta} (d_\delta^* -
d_\delta (x^*)) = \frac{2}{\delta} \left( d(x^*_\delta)  - \frac{\delta}{2} \|x^*_\delta - x^0\|^2 - d(x^*) + \frac{\delta}{2} \|x^* - x^0\|^2 \right) \\
& \leq \| x^* - x^0\|^2.
\end{align*}
Based on the previous inequality we can bound  $\| x^0 - x^*_\delta
\|$ as follows:
\begin{align*}
\| x^0 - x^*_\delta \| \leq \|x^0 - x^*\| + \|x^* - x^*_\delta \|
\leq 2 \|x^0 - x^* \| = 2 \mathcal{R}_{\text{d}}.
\end{align*}

\noindent Thus, the number of iterations $k$ we need in order to
attain $\epsilon^2$ accuracy, i.e. $d_\delta^* - d_\delta(x^k) \leq
\epsilon^2$,  is given by:
\begin{align}
\label{logeps_cone} k =2 \sqrt{\frac{L_\text{d}+\delta}{\delta}}
\log \left( \frac{  \mathcal{R}_{\text{d}} \sqrt{2(L_\text{d} +
2\delta)}}{\epsilon} \right).
\end{align}
Since $d_\delta^* = d_\delta (x_\delta^*) \geq d_\delta(x^*) = f^* -
\frac{\delta}{2} \|x^*  - x^0\|^2$ and $d_\delta (\cdot) \leq
d(\cdot)$, we get that after the number of iterations
\eqref{logeps_cone} we have from $d_\delta^* - d_\delta(x^k) \leq
\epsilon^2$ that:
\[ f^* - d(x^k) \leq \frac{\delta}{2} \|x^* - x^0\|^2 + \epsilon^2 = \frac{\delta}{2} \mathcal{R}_{\text{d}}^2  + \epsilon^2.
\]
Let us assume  for simplicity that $\epsilon^2 \leq \epsilon/2$ and
choose:
\begin{align}
\label{delta_cone} \delta =
\frac{\epsilon}{\mathcal{R}_{\text{d}}^2}.
\end{align}
Then, we get an estimate on dual suboptimality for the  dual problem
of  \eqref{eq_prob_princc}:
\[ f^* - d(x^k) \leq \epsilon.  \]

\noindent We are now ready to prove one the main results of this
paper:
\begin{theorem}
\label{th_dfglast_cone}  Let Assumption \ref{as_strong} hold and the
sequences $\left(x^k, y^k, u^k\right)_{k\geq 0}$ be generated by the
modified  Algorithm {\bf (DFG)}. Then, for a given accuracy
$\epsilon>0$ we get an $\epsilon$-primal solution for
\eqref{eq_prob_princc} in the last primal  iterate $v^{k}$ of
modified  Algorithm \textbf{(DFG)} after $k =
\mathcal{O}(\frac{1}{\sqrt{\epsilon}} \log(\frac{1}{\epsilon}))$
iterations.
\end{theorem}

\begin{proof}
\noindent First,  we determine a bound on $\|x^k - x^*_\delta\|$.
 Since $d_\delta (\cdot)$ is $\delta$-strongly concave function and
$d_\delta^* - d_\delta(x^k) \leq   \epsilon^2$, we have:
\begin{align}
\label{bound_delta_cone} \|x^k - x^*_\delta \|^2 \leq
\frac{2}{\delta} (d_\delta^* - d_\delta (x^k)) \leq \frac{2
\epsilon^2}{\delta}.
\end{align}
Moreover, since  the dual gradient $\nabla d_{\delta}$ is Lipschitz,
it    satisfies \cite{Nes:04}:
\[ d^*_{\delta} \geq d_\delta \left( \left[ x^k + \frac{1}{L_{d,\delta}} \nabla d_\delta(x^k) \right]_{\mathcal{K}^*} \right)
\geq d_\delta (x^k) + \frac{L_{d,\delta}}{2} \|\nabla^+ d_\delta
(x^k)\|^2.
\]
Using that $d_\delta^* - d_\delta(x^k) \leq   \epsilon^2$ in the
previous inequality we obtain:
\[  \|\nabla^+ d_\delta (x^k)\|^2 \leq 2 L_{d,\delta} \epsilon^2.  \]
\noindent Now using the previous bound on $\|\nabla^+ d_\delta
(x^k)\|$, the fact that $v^k = u(x^k)$ and  the expressions  of
$\nabla d_\delta(x) = \nabla d(x) - \delta (x - x^0)$ and $\nabla
d(x) = - G u(x) - g$ we get an estimate on primal infeasibility in
the last primal iterate $v^k$:
\begin{align*}
 d_{\ca}(G v^k + g) &  = d_{\ca}(- \nabla d(x^k)) \le \| - \nabla d(x^k) -  \left[\nabla d_{\delta}(x^k) \right]_{\ca}\| \\
& \le d_{\ca}(- \nabla d_\delta (x^k)) + \delta \|x^k - x^0\|
\overset{\eqref{decrease_ggm}}{\le} L_{\text{d},\delta} \| \nabla^+
d_\delta (x^k)\| + \delta \|x^k - x^0\| \\
& \leq \sqrt{2 L_{d,\delta}^3} \epsilon + \delta (\|x^k -
x^*_\delta\| + \|x^0 - x^*_\delta\|)
\overset{\eqref{bound_delta_cone}}{\leq} \sqrt{2 L_{d,\delta}^3}
\epsilon + \epsilon \sqrt{2 \delta} + 2 \delta  R_\text{d} \\
& \overset{\eqref{delta_cone}}{\leq} 4\epsilon
\left(\sqrt{L_\text{d}^3} + \frac{1}{\mathcal{R}_{\text{d}}}
\right).
\end{align*}

\noindent In order to derive a convergence estimate  on  primal
suboptimality, we observe that  for any $x \in \mathcal{K}^*$, the
optimality conditions of the inner subproblem are given by:
\begin{equation}\label{optcond_inner}
 \langle \nabla f(u(x))  - G^Tx, u  - u(x)\rangle \ge 0 \qquad \forall u \in U.
\end{equation}
Since $d_{\delta}(\cdot)$ is concave and has Lipschitz continuous
gradient, it satisfies \cite{Nes:04}:
\begin{equation*}
 d_{\delta}(x) \le d_{\delta}(y) + \langle \nabla d_{\delta}(y), x - y\rangle
  - \frac{1}{2L_{\text{d},\delta}}\norm{\nabla d_{\delta}(x) - \nabla d_{\delta}(y)}^2
 \quad \forall x,y \in \ca^*.
\end{equation*}
Taking into account the expression for $\nabla d_{\delta}$, using $y
= x^*_{\delta}$ in the previous inequality and the optimality
conditions for $x^*_{\delta}$, we have:
\begin{equation}\label{diff_grad}
 \norm{\nabla d(x) - \nabla d(x^*_{\delta})} - \delta \norm{x-x^*_{\delta}} \le
 \norm{\nabla d_{\delta}(x) - \nabla d_{\delta}(x^*_{\delta})}
 \le \sqrt{2(L_{d}+\delta)(d^*_{\delta} - d_{\delta}(x))}.
\end{equation}

\noindent On the other hand, using the strong convexity of the
function $f$ and taking $u = u^*_{\delta} = u(x^*_{\delta})$ in
\eqref{optcond_inner}, we obtain:
\begin{equation*}
f(u(x)) - f(u^*_{\delta}) + \frac{\sigma}{2}\norm{u^*_{\delta} -
u(x)}^2 \le \langle x, G u(x) - G u^*_{\delta}\rangle.
\end{equation*}
Adding in both sides the term $\langle x^*_{\delta}, G u^*_{\delta}
+ g\rangle + \frac{\delta}{2}\norm{x^*_{\delta}-x^0}^2$, we get:
\begin{align*}
f(u(x)) - d^*_{\delta}
&\le \langle x, G u(x) - G u^*_{\delta}\rangle + \langle x^*_{\delta}, G u^*_{\delta} + g\rangle + \frac{\delta}{2}\norm{x^*_{\delta}-x^0}^2\\
& = \langle x, \nabla d(x^*_{\delta}) - \nabla d(x) \rangle +
\langle x^*_{\delta}, G u^*_{\delta} + g\rangle +
\frac{\delta}{2}\norm{x^*_{\delta}-x^0}^2.
\end{align*}
Taking $x = 0$ in \eqref{optcond_regdelta}, using \eqref{diff_grad},
Lipschitz gradient property of $d(\cdot)$,the fact that
$d_{\delta}^* \le f^*$, the Cauchy-Schwartz inequality and previous
inequality, we obtain:
\begin{align*}
&f(u(x)) - f^* \le f(u(x)) - d^*_{\delta}
\le \langle x, \nabla d(x^*_{\delta}) - \nabla d(x) \rangle + \langle x^*_{\delta}, G u^*_{\delta} + g\rangle + \frac{\delta}{2}\norm{x^*_{\delta}-x^0}^2\\
&\le \norm{x -x^*_{\delta}} \norm{ \nabla d(x^*_{\delta}) - \nabla
d(x)} + \norm{x^*_{\delta}} \norm{ \nabla d(x^*_{\delta}) - \nabla
d(x)} +
\langle x^*_{\delta}, G u^*_{\delta} + g\rangle + \frac{\delta}{2}\norm{x^*_{\delta}-x^0}^2\\
& \overset{\eqref{diff_grad}}{\le} L_{\text{d}}\norm{x -x^*_{\delta}}^2  + \norm{x^*_{\delta}}(\sqrt{2(L_{d}+\delta)(d^*_{\delta} - d_{\delta}(x))} + \delta\norm{x^0-x^*_{\delta}}) \\
& \hspace{8cm} + \langle x^*_{\delta}, G u^*_{\delta} + g\rangle + \frac{\delta}{2}\norm{x^*_{\delta}-x^0}^2\\
& \overset{\eqref{optcond_regdelta}}{\le} L_{\text{d}}\norm{x
-x^*_{\delta}}^2 + \norm{x^*_{\delta}}
\left(\sqrt{2(L_{d}+\delta)(d^*_{\delta} - d_{\delta}(x))} +
\delta\norm{x^0-x^*_{\delta}}\right) \\
& \hspace{6.5cm} + \frac{\delta}{2}\left( \norm{x^0}^2 - \norm{x^*_{\delta}}^2 - \norm{x^*_{\delta} - x^0}^2 \right) + \frac{\delta}{2}\norm{x^*_{\delta}-x^0}^2\\
& \le \frac{2L_{\text{d}}}{\sigma} (d^*_{\delta} - d_{\delta}(x)) +
\norm{x^*_{\delta}} \left(\sqrt{2(L_{d}+\delta)(d^*_{\delta} -
d_{\delta}(x))} + \delta\norm{x^0-x^*_{\delta}}\right)
+\frac{\delta}{2}\norm{x^0}^2.
\end{align*}
On the other hand, we have:
\begin{align*}
f^* &=  \min_{u \in U} f(u) + \langle x^*, g(u) \rangle \leq f(u(x))
+ \langle x^*, g(u(x)) \rangle \nonumber\\
& =  f(u(x)) + \langle x^*, -G u(x) - g \rangle \leq f(u(x))
+ \langle x^*,  [G u(x) + g]_{\ca} - (G u(x) + g) \rangle \nonumber\\
& \leq f(u(x)) + \norm{x^*} \text{dist}_{\mathcal{K}}(G u(x) + g).
\end{align*}

\noindent Therefore, using \eqref{delta_cone} and the facts that
$\epsilon^2 \le \frac{\epsilon}{2}$ and $v^k=u(x^k)$, we derive the
convergence rate for primal suboptimality from the previous
estimates on suboptimality and infeasibility:
\begin{align*}
 - 4\epsilon \left(R_{\text{d}} + \norm{x^0}\right) \left(\sqrt{L_\text{d}^3} +
\frac{1}{\mathcal{R}_{\text{d}}} \right) \le  f(v^k) - f^* \le
\epsilon C_r,
\end{align*}
where $C_r = \frac{L_{\text{d}}}{\sigma} + \left(2 R_{\text{d}} +
\norm{x^0} \right) \left(\sqrt{2(L_{\text{d}} + \delta)} +
\frac{2}{R_{\text{d}}} \right) +
\frac{\norm{x^0}^2}{R_{\text{d}}^2}$.

\vspace{5pt}

\noindent Now, if we replace the expression  for $\delta$ from
\eqref{delta_cone} in the expression of $k$ from \eqref{logeps_cone}
it follows that we obtain $\epsilon$-accuracy for primal
suboptimality and infeasibility in the last primal iterate $v^k$ for
the modified Algorithm (\textbf{DFG})  after $k =
\mathcal{O}(\frac{1}{\sqrt{\epsilon}} \log(\frac{1}{\epsilon}))$
iterations.
\end{proof}

\noindent From Theorem \ref{th_dfglast_cone} it follows that we
obtain $\epsilon$-accuracy for primal suboptimality and
infeasibility for the modified Algorithm (\textbf{DFG}) in the last
primal iterate $v^{k}$ after $k =
\mathcal{O}(\frac{1}{\sqrt{\epsilon}} \log(\frac{1}{\epsilon}))$
iterations which is  better than $\mathcal{O}(\frac{1}{\epsilon})$
iterations obtained in Theorem \ref{th_dfglast}  for the last primal
iterate $v^k$ or in \cite{BecNed:14}. From our knowledge Theorem
\ref{th_dfglast_cone} provides the best convergence rate  for dual
fast gradient method in the last iterate.  However, the modified
algorithm needs to know the parameter $\delta$, that according to
\eqref{delta_cone}, is depending on $\mathcal{R}_{\text{d}}$.  In
practice, we need to know an  estimate of $\mathcal{R}_{\text{d}}$.


\subsection{Sublinear convergence in an average  primal sequence}
\label{sec_dfgav} \noindent In this section we derive sublinear
estimates for primal infeasibility and suboptimality of the average
primal sequence $(\hat u^k)_{k \ge 0}$ as defined in \eqref{av_dfg2}
for Algorithm \textbf{(DFG)}. 

\begin{theorem}
\label{th_dfgav} Let Assumption \ref{as_strong} hold and the
sequences $\left(x^k, y^k, u^k\right)_{k\geq 0}$ be generated by
Algorithm {\bf (DFG)}.  Then, for a given accuracy $\epsilon>0$ we
get an $\epsilon$-primal solution for \eqref{eq_prob_princc} in the
average primal iterate  $\hat u^k$ of Algorithm \textbf{(DFG)} after
$k = {\mathcal O} (\frac{1}{\sqrt{\epsilon}})$ iterations.
\end{theorem}

\begin{proof}  For any $j \ge 0$ we have $\left[ y^{j} +
\frac{1}{L_{\text{d}}}\nabla d(y^{j}) \right]_{\mathcal{K}^*}
=x^{j}.$ Let us  denote $z^j = y^{j} + \frac{1}{L_{\text{d}}}\nabla
d(y^{j})$.   Then, we can write as follows:
\begin{align} \label{feasibility_aux1}
\theta_j & \left( [z^j]_{\mathcal{K}^*} - z^j + \frac{1}{L_{\text{d}}}\nabla  d(y^j) \right) = \theta_j \left(\left[ y^{j} + \frac{1}{L_{\text{d}}}\nabla d(y^{j}) \right]_{\mathcal{K}^*} - y^{j} \right) \nonumber\\
& =  \theta_{j} (x^{j} - y^{j})  = \theta_{j}(x^{j} - x^{j-1}) + (\theta_{j-1} -1)(x^{j-2} - x^{j-1})\nonumber\\
& = \underbrace{x^{j-1} + \theta_{j}(x^{j}-x^{j-1})}_{w^{j}} -
\underbrace{(x^{j-2} + \theta_{j-1}(x^{j-1}-x^{j-2}))}_{w^{j-1}}.
\end{align}

\noindent Note that  $\nabla d(y^j) = - G u^j - g$.  Further, summing on the history,  multiplying by $\frac{L_{\text{d}}}{S_k^\theta}$ the previous relation and using  the definition of $\hat u^k$, we obtain:
\begin{align*}
\frac{L_\text{d}}{S_k^\theta} (w^{k}-w^0) & = L_{\text{d}}\sum\limits_{j=0}^k  \frac{\theta_j}{S_k^\theta}([z^j]_{\mathcal{K}^*} - z^j) +
\sum\limits_{j=0}^k  \frac{\theta_j}{S_k^\theta}  \nabla d(y^j) \\
& = L_{\text{d}}\sum\limits_{j=0}^k  \frac{\theta_j}{S_k^\theta}([z^j]_{\mathcal{K}^*} - z^j) - (G \hat u^k + g).
\end{align*}
Since   $[z^j]_{\mathcal{K^*}} - z^j \in \mathcal{K}$ (according to \eqref{prop_proj}), we have $L_{\text{d}}\sum\limits_{j=0}^k  \frac{\theta_j}{S_k^\theta}([z^j]_{\mathcal{K}^*} - z^j) \in \ca$. In conclusion, using the definition of the distance, we obtain:
\begin{align*}
d_{\mathcal{K}}(G\hat{u}^k + g) & \le \left\| L_{\text{d}}\sum\limits_{j=0}^k  \frac{\theta_j}{S_k^\theta}([z^j]_{\mathcal{K}^*} - z^j) - (G \hat u^k + g)  \right\| \\
&= \frac{L_\text{d}}{S_k^\theta}\norm{w^{k}-w^0} \le
\frac{4L_\text{d}}{(k+1)^2} \norm{w^k -w^{0}}.
\end{align*}

\noindent Taking $x=x^*$ in \eqref{corr2_relation} and using
 that the two  terms $\theta_{k+1}(f^*-d(x^{k+1}))$ and $\sum\limits_{i=1}^{k+1}\theta_{i}
 \Delta(x^*,y^i)$ are positive,
we get $\| w^k - x^* \| \le \| x^0 - x^* \|$ for all $k \geq 0$. Moreover, we have $\|w^0 - x^*\| = \| x^0 - x^*\|$. Thus, we can further bound the primal infeasibility as follows:
\begin{align}
\text{dist}_{\mathcal{K}}(G \hat{u}^k +g)    & \le
\frac{4L_\text{d}}{(k+1)^2} \|w^k - w^0\|
 \le \frac{4L_\text{d}}{(k+1)^2}(\|w^k - x^* \| + \|w^0 - x^*\|)  \nonumber\\
& \le \frac{8L_\text{d}}{(k+1)^2} \|x^0 - x^*\| = \frac{8L_\text{d}
\mathcal{R}_{\text{d}}}{(k+1)^2}.\label{infes_av}
\end{align}

\noindent Further, we derive sublinear  estimates
for primal suboptimality. First, note that:
\begin{align*}
\Delta(x,y^{k+1}) &=  d(y^{k+1}) + \langle \nabla d(y^{k+1}),  x-y^{k+1}\rangle -d(x) \\
&=\mathcal{L}(u^{k+1},y^{k+1}) + \langle g(u^{k+1}), x-y^{k+1}\rangle -d(x) \\
&=f(u^{k+1}) + \langle g(u^{k+1}), x\rangle -d(x) =
\mathcal{L}(u^{k+1},x) -d(x).
 \end{align*}
\noindent Summing on the history  and using  the convexity of
$\mathcal{L}(\cdot,x)$, we get:
\begin{align}
\sum\limits_{i=1}^{k+1}\theta_i \Delta(x,y^i)
&= \sum\limits_{i=1}^{k+1}\theta_i(\mathcal{L}(u^{i},x) -d(x))\nonumber\\
&\ge S_{k+1}^{\theta}\left(\mathcal{L}(\hat{u}^{k+1},x) -d(x)\right)
= \theta_{k+1}^2\left(\mathcal{L}(\hat{u}^{k+1},x) -d(x)\right).
\label{sum_theta_aux}
\end{align}
Using \eqref{sum_theta_aux} in \eqref{corr2_relation}, and dropping
the term $L_\text{d}/2\norm{w^{k+1}-x}^2$, we have:
\begin{equation}\label{}
f(\hat{u}^{k+1}) + \langle x, g(\hat{u}^{k+1})\rangle - d(x^{k+1})
\le \frac{L_\text{d}}{2\theta_{k+1}^2}\norm{x^0-x}^2 \quad \forall x
\in {\ca^*}.
\end{equation}

\noindent Taking  $x=0 \in {\ca^*}$ in the previous inequality, we get:
\begin{align*}
f(\hat{u}^{k+1}) - d(x^{k+1}) \leq
\frac{L_\text{d}}{2\theta_{k+1}^2}\norm{x^0}^2 \leq
\frac{2L_\text{d}}{(k+2)^2} \norm{x^0}^2.
\end{align*}
Taking in account that $d(x^{k}) \le f^*$, then we have:
\begin{equation}
\label{subopt_av_1} f(\hat{u}^{k}) - f^* \le
\frac{2L_\text{d}}{(k+1)^2} \norm{x^0}^2.
\end{equation}

\noindent On the other hand, we have:
\begin{align}
f^* &=  \min_{u \in U} f(u) + \langle x^*, g(u) \rangle \leq
f(\hat u^k) + \langle x^*, g(\hat u^k)  \rangle \nonumber\\
& \le f(\hat u^k) + \norm{x^*} \text{dist}_{\mathcal{K}} (G \hat u^k+ g)
\overset{\eqref{infes_av}}{\leq} f(\hat u^k) + \frac{8L_\text{d}
\mathcal{R}_{\text{d}}}{(k+1)^2}   \norm{x^*}. \label{subopt_av_2}
\end{align}
From \eqref{subopt_av_1} and \eqref{subopt_av_2} we obtain an
estimate on primal suboptimality:
\begin{equation}
\label{ps_dfg} |f(\hat{u}^{k}) - f^*| \le
\frac{8L_\text{d}}{(k+1)^2}\left[ \mathcal{R}_{\text{d}}^2 +
\max(\mathcal{R}_{\text{d}},\norm{x^0})^2 \right].
\end{equation}

\noindent Thus,  we have obtained sublinear estimates  of order
$\mathcal{O}(\frac{1}{k^2})$ for primal infeasibility (inequality
\eqref{infes_av}) and primal suboptimality (inequality
\eqref{ps_dfg}) for the average primal  sequence $(\hat{u}^k)_{k
\geq 0}$ generated by Algorithm \textbf{(DFG)}.   Now, if we want to
get an $\epsilon$-primal solution in $\hat{u}^{k}$ we need to
perform $k = {\mathcal O} (\frac{1}{\sqrt{\epsilon}})$ iterations.
\end{proof}

\noindent Based on \eqref{distance}, we can also characterize the
distance from $\hat u^k$ to the unique primal optimal solution
$u^*$. Using  \eqref{infes_av} and \eqref{subopt_av_1}, we get:
\[   \| \hat u^k - \bo{u}^*\|^2  \leq
\frac{1}{(k+1)^2} \left[ \frac{4 L_\text{d}}{\sigma_{\mathrm{f}}} \|x^0\|^2 + \frac{16 L_\text{d} \mathcal{R}_\text{d}}{\sigma_{\mathrm{f}}}
 (\mathcal{R}_\text{d} + \|x^0\|) \right]. \]


\noindent In Theorem \ref{th_dfglast_cone}  we obtained an
$\epsilon$-primal solution for the modified Algorithm (\textbf{DFG})
in the last primal iterate $v^{k}$ after $k =
\mathcal{O}(\frac{1}{\sqrt{\epsilon}} \log(\frac{1}{\epsilon}))$
iterations, which  is of the same order (up to a logarithmic term)
as for the primal average sequence  $\hat u^k$ from previous Theorem
\ref{th_dfgav}. Moreover, the reader should also notice that all our
previous  convergence estimates depend only on three  constants: the
Lipschitz constant $L_\text{d}$, the initial starting dual point
$x^0$ and its distance to the dual optimal  set denoted
$\mathcal{R}_\text{d}$. Moreover, if $x^0=0$, then $f(\hat u^k) -
f^* \leq 0$, i.e. the function values in  the primal average
sequences  are always below the optimal value for Algorithms
\textbf{(DG)} and \textbf{(DFG)}.


\section{Dual error bound property and linear convergence of dual first order methods}
\label{sec_lin}
In this section, we show that if the dual problem has an error bound type property we can get an $\epsilon$-primal solution for problem \eqref{eq_prob_princc}  with the previous dual first order methods in $k = \mathcal{O}(\log(\frac{1}{\epsilon}))$ iterations. Thus, in this section we assume that the dual problem of \eqref{eq_prob_princc} has an error bound property. More precisely, we assume that  for any $M>0$ there exists a constant  $\kappa > 0$ depending on $M$ and the data of problem \eqref{eq_prob_princc} such that the following
error bound property holds for the  corresponding dual problem of \eqref{eq_prob_princc}:
\begin{equation}
\label{eq_error_bound} \|x - \bar{x}\| \leq \kappa \|\nabla^+
d(x)\| ~~~ \forall x \in {\ca^*}, \; f^* - d(x) \leq M,
\end{equation}
where  $\bar{x}=\left[x\right]_{X^*}$ (i.e. the Euclidean projection
of $x$ onto the optimal dual set $X^*$)  and recall that  $\nabla^+ d(x)$ denotes the gradient map: $\nabla^+ d(x)= [x +  \frac{1}{L_\text{d}} \nabla
d(x)]_{+} - x$.

\begin{remark}
For example, if we consider  a linearly constrained convex  problem ($\ca = \rset^p_{-}$):
\begin{align}
\label{urn} \min_{u \in \rset^n} f(u): \quad \text{s.t.} \quad Gu +
g \leq 0,
\end{align}
where we assume that $f$  is $\sigma_\text{f}$-strongly convex
function and  has $L_{\text{f}}$-Lipschitz continuous
gradient, $U = \rset^n$ and $G \in \rset^{p \times n}$, then   in \cite{LuoTse:92a,NecNed:14a,WanLin:13} it has been
proved that the corresponding dual problem satisfies an error bound type property. Indeed,  for the convex function $f$, we denote its conjugate by \cite{RocWet:98}: $ \tilde{f}(y) =    \max \limits_{x \in \rset^{n}} \langle y, x \rangle - f(x)$. According to Proposition 12.60 in \cite{RocWet:98}, under the previous assumptions,  function $\tilde{f}(y)$ is strongly
convex w.r.t. Euclidean norm, with constant
$\sigma_{\tilde{\mathrm{f}}}= \frac{1}{L_{{\mathrm{f}}}}$ and has
Lipschitz continuous gradient with constant $L_{\tilde{\mathrm{f}}}=
\frac{1}{\sigma_{{\mathrm{f}}}}$. Note that in these settings our
dual function of \eqref{urn} can be written as: $d(x)=-\tilde{f}(-G^T x)-g^T x$.  Since $f$ is strongly convex, the dual gradient $\nabla d(x) =  G u(x) + g$ is Lipschitz continuous with constant $L_\text{d} = \frac{\|G\|^2}{\sigma_\text{f}}$
\cite{Nes:05}. Furthermore, if  $G$ has full row rank, then it
follows immediately that the dual function $d$ is strongly convex.
Therefore,  we consider  the nontrivial case when $G$ is rank
deficient. In \cite{LuoTse:92a,NecNed:14a,WanLin:13} it has been  proved that
for convex problem \eqref{urn} with
function $f$ being  $\sigma_{\text{f}}$-strongly convex and  having
$L_{\text{f}}$-Lipschitz gradient and $U = \rset^n$,  for any $M>0$ there exists a constant  $\kappa > 0$ depending
on $M$ and the data of problem \eqref{urn} such that an
error bound property of the form \eqref{eq_error_bound} holds for the  corresponding dual problem. \qed
\end{remark}

\noindent Next, we derive a strong convex like inequality that will be used in the sequel.
\begin{theorem}
Under Assumption \eqref{as_strong} and the error bound property \eqref{eq_error_bound} for the  corresponding dual  of convex  problem \eqref{eq_prob_princc} the following inequality holds:
\begin{align}
\label{slr}
f^*  - d(x) \geq \frac{L_\text{d}}{2 \kappa^2} \| x - \bar{x}\|^2   \quad \forall x \in {\ca^*}, \; f^* - d(x) \leq M.
\end{align}
\end{theorem}

\begin{proof}
Let us define $x^+ = [x + 1/L_\text{d} \nabla d(x)]_{\ca^*}$ so that $\nabla^+ d(x) = x^+ - x$.  Note that $x^+$ is the optimal solution of the following convex problem:
\begin{align}
\label{op_g} x^+ = \arg \min_{z \in {\ca^*}} d(x) + \langle \nabla d(x),
z-x \rangle - \frac{L_\text{d}}{2} \|z - x\|^2.
\end{align}
From \eqref{eq_descent} and  the optimality conditions of
\eqref{op_g} we get the following increase in terms of the objective
function $d$:
\begin{align}
\label{decrease_g} d(x^+) & \geq d(x) + \langle \nabla d(x), x^+ - x
\rangle - \frac{L_\text{d}}{2} \| x^+ - x\|^2 \geq d(x) + \frac{L_\text{d}}{2} \|
x^+ - x\|^2.
\end{align}
Combining  \eqref{eq_error_bound} and  \eqref{decrease_g} we obtain:
\begin{align*}
\| x - \bar{x}\|^2 \leq \frac{2 \kappa^2}{L_\text{d}} (d(x^+) - d(x)) \leq
\frac{2 \kappa^2}{L_\text{d}} (f^* - d(x))   \quad \forall x \in {\ca^*}, \; d(x) \geq f^* - M,
\end{align*}
which shows the statement of the theorem.
\end{proof}

\noindent Firstly, we consider Algorithm (\textbf{DG})). For simplicity, we assume
constant step size $\alpha_k = \frac{1}{L_\text{d}}$. Since
Algorithm (\textbf{DG}) is an ascent method according to
\eqref{dg_pr2}, we can take $M = f^* -d(x^0)$. Thus, the error bound
property \eqref{eq_error_bound} holds for the sequence $(x^k)_{k
\geq 0}$ generated by  Algorithm (\textbf{DG}), i.e. there exists
$\kappa > 0$ such that:
\begin{equation}
\label{global_error_bound} \|x^k - \bar{x}^k\| \leq \kappa
\|\nabla^+ d(x^k)\| = \kappa \|x^{k+1} - x^k\| \quad \forall k \geq
0,
\end{equation}
where  $\bar{x}^k  = \left[x^k\right]_{X^*}$.  The following theorem
provides an estimate on the dual suboptimality for Algorithm ({\bf
DG}) with constant step size.

\begin{theorem}
\label{theorem_dual_optim_dg} Under Assumption \eqref{as_strong} and the error bound property \eqref{eq_error_bound} for the  corresponding dual  of problem  \eqref{eq_prob_princc},  the sequence $\left(x^k\right)_{k\geq 0}$  generated
by Algorithm ({\bf DG}) converges linearly in terms of the distance to the dual optimal set $X^*$ and of the  dual objective function values:
\begin{equation}
\label{bound_dual_optim_dg}
\| x^k - \bar{x}^k \| \leq \left( \frac{\kappa}{\sqrt{1+\kappa^2}}\right)^{k} \mathcal{R}_\text{d} \quad
\text{and} \quad f^* - d({x}^{k}) \leq \frac{\Ld
\mathcal{R}^2_\text{d}}{2}
\left(\frac{\kappa^2}{1+\kappa^2}\right)^{k-1} \quad \forall k \geq 0.
\end{equation}
\end{theorem}

\begin{proof}
From \eqref{dg_prr} and  concavity of $d$, we get:
\[ \| x^{k+1} - x \|^2 \leq \| x^k - x \|^2 + \frac{2}{L_\text{d}}
 \left( d(x^{k+1}) - d(x) \right)  \qquad  \forall x \in {\ca^*}. \]  Taking now in the previous relations $x =\bar{x}^{k}$ and using
$\|x^{k+1} - \bar{x}^{k+1}\| \leq \|x^{k+1} - \bar{x}^k\|$ and the strong convex like inequality  \eqref{slr}, we get:
\[  \|x^{k+1} - \bar{x}^{k+1}\|^2 \leq \|x^{k} - \bar{x}^{k}\|^2 - \frac{1}{\kappa^2}  \|x^{k+1} - \bar{x}^{k+1}\|^2, \]
or equivalently
\begin{align}
\label{lin_convw}
\|x^{k+1} - \bar{x}^{k+1}\|  \leq
\frac{\kappa}{\sqrt{1+\kappa^2}} \|x^{k} -
\bar{x}^{k}\|.
\end{align}
\noindent Thus, we obtain  linear convergence rate in terms of distance to the optimal set $X^*$:
\begin{align}
\label{lin_convs} \|x^{k} - \bar{x}^{k}\|  \leq \left( \frac{\kappa}{\sqrt{1+\kappa^2}}\right)^{k} \|x^{0} - \bar{x}^{0}\| = \left( \frac{\kappa}{\sqrt{1+\kappa^2}}\right)^{k} \mathcal{R}_\text{d}.
\end{align}

\noindent We can also derive  linear convergence  in terms of dual
function values:
\begin{align*}
d(x^{k+1}) & \overset{\eqref{eq_descent}}{\geq}  d(x^k) + \langle \nabla d(x^k),
x^{k+1} - x^k\rangle - \frac{L_\text{d}}{2} \|x^{k+1} - x^k\|^2 \\
& =  \max_{x \in {\ca^*}} \; d(x^k) + \langle \nabla d(x^k), x - x^k\rangle
- \frac{L_\text{d}}{2} \|x - x^k\|^2 \\
& \geq  \max_{x \in {\ca^*}} \; d(x) - \frac{L_\text{d}}{2} \|x -x^k\|^2
\geq d(\bar{x}^k)  - \frac{L_\text{d}}{2} \|x^k - \bar{x}^k\|^2 \\
& \overset{\eqref{lin_convs}}{\geq} f^* - \frac{L_\text{d}\mathcal{R}_\text{d}^2}{2}
\left( \frac{\kappa}{\sqrt{1+\kappa^2}}\right)^{2k}.
\end{align*}
\end{proof}

\noindent Note that our proof from Theorem
\ref{theorem_dual_optim_dg} is different from Tseng's proof
\cite{Tse:10} for linear convergence of gradient method under an
error bound property. More precisely, in our proof we make use
explicitly of the strong convex like inequality \eqref{slr} which
allows us to get for $\|x^k - \bar x^k\|$ better  convergence rate
than in \cite{Tse:10}.

\noindent We now derive linear estimates for primal infeasibility and
primal suboptimality for the last iterate sequence $(u^k)_{k \geq
0}$ generated by our Algorithm ({\bf DG}) with constant step size
$\alpha_k = \frac{1}{L_\text{d}}$. For simplicity of the exposition
let us denote:
\[   c_1 = \frac{\Ld \mathcal{R}_\text{d}^2}{2} \quad \text{and}
\quad \theta=\frac{\kappa^2}{1+\kappa^2}. \]
Clearly, $\theta < 1$.  From Theorem \eqref{theorem_dual_optim_dg}
we obtain:
\begin{equation}
\label{conv_lineb} f^* - d({x}^{k}) \leq c_1 \theta^{k-1}.
\end{equation}

\begin{theorem}
\label{th_urndglast} Under the assumptions of Theorem \ref{theorem_dual_optim_dg},
let the sequences $\left(x^k, u^k\right)_{k\geq 0}$ be generated by
Algorithm {\bf (DG)}. Then, for a given accuracy $\epsilon>0$ we get
an $\epsilon$-primal solution for \eqref{eq_prob_princc} in the last primal
iterate $u^{k}$ of Algorithm \textbf{(DG)} after $k =
\mathcal{O}(\log(\frac{1}{\epsilon}))$ iterations.
\end{theorem}

\begin{proof}
Combining \eqref{ineq_x} and \eqref{conv_lineb} we obtain the
following  relation:
\begin{align}
\label{dist_sublineb} \| u^k - u^* \| \leq
\sqrt{\frac{2c_1}{\sigma_\text{f}}} \theta^{\frac{k-1}{2}} = \sqrt{\frac{\Ld \mathcal{R}_\text{d}^2}{\sigma_\text{f}}} \theta^{\frac{k-1}{2}}.
\end{align}

\noindent Then, combining the previous relation
\eqref{dist_sublineb} and \eqref{ineq_feas2} we obtain a linear
estimate for feasibility violation of the last iterate $u^k$:
\begin{align}
\label{fes_sublineb} \text{dist}_{\ca}(G u^k + g) & \leq   \|G\|  \| \bo{u}^k - u^* \|
  \leq  \|G\| \sqrt{\frac{2c_1}{\sigma_\text{f}}}
\theta^{\frac{k-1}{2}}    \leq  L_\text{d} \mathcal{R}_\text{d}
\theta^{\frac{k-1}{2}},
\end{align}
where we used the definitions of  $L_\text{d} =
\|G\|^2/\sigma_\text{f}$ and $c_1$. Finally, we derive linear
estimates for primal suboptimality of the last iterate $u^k$.
Combining \eqref{dist_sublineb} and \eqref{ineq_opt} we obtain:
\begin{align}
\label{opt_sublineb} |f(u^k) - f^*| & \leq   ( \| x^k - x^*\| +
\|x^*\|) \|G\|   \sqrt{\frac{2c_1}{\sigma_\text{f}}} \theta^{\frac{k-1}{2}} \nonumber \\
&  \leq  (2  \mathcal{R}_\text{d} + \|x^0\|) L_\text{d}
\mathcal{R}_\text{d} \theta^{\frac{k-1}{2}}.
\end{align}
In conclusion, we have obtained linear estimates  of order
$\mathcal{O}(\theta^k)$, with $\theta<1$, for primal infeasibility
(inequality \eqref{fes_sublineb}) and  suboptimality (inequality
\eqref{opt_sublineb}) for the last iterate  sequence $(u^k)_{k \geq
0}$ generated by  Algorithm \textbf{(DG)}. Now, if we want to get an
$\epsilon$-primal solution in $u^{k}$ we need to perform $k =
\mathcal{O}(\log(\frac{1}{\epsilon}))$ iterations.
\end{proof}

\noindent Secondly, we  show that under Assumption \eqref{as_strong}
and the error bound property \eqref{eq_error_bound} for the
corresponding dual  of problem  \eqref{eq_prob_princc}, a restarting
version of Algorithm {\bf (DFG)} has linear convergence. Similar to
Algorithm (\textbf{DG}), we can also take in this  case  $M = f^*
-d(x^0)$ and thus the error bound property \eqref{eq_error_bound}
holds for the sequence $(x^k)_{k \geq 0}$ generated by  a restarting
version of Algorithm (\textbf{DFG}). Indeed, combining
\eqref{bound_dual_optim_dfg} and \eqref{slr} we get:
\[ f^* - d(x^k)   \leq \frac{2 L_\text{d}}{(k+1)^2} \| x^0 - \bar{x}^0\|^2 \leq
 \frac{4 \kappa^2}{(k+1)^2} (f^*  - f(x^0)) = c^2 (f^* - f(x^0)), \]
where we choose a positive constant $c \in (0, \; 1)$ such that
\[  c = \frac{2 \kappa }{k+1}. \]
Then, for fixed $c$, the number of iterations $K_c$ that we need to
perform in order to obtain $f^*  - d(x^{K_c})  \leq c^2 ( f^* - d(x^0)) $
is given by:
\[ K_c = \left  \lfloor \frac{2 \kappa}{c} \; \right \rfloor.  \]
Note that if the optimal value $f^*$ is known in advance, then we just need to
restart  Algorithm {\bf (R-DFG)} at iteration $K_c^* \leq
K_c$ when the following condition holds:
\[   f^* - d(x^{K_c^*,j})  \leq c^2 (f^* - d(x^{0,j})),  \]
which can be practically   verified. After each $K_c$ steps of Algorithm {\bf (DFG)} we
restart it obtaining the following scheme:
\begin{center}
\framebox{
\parbox{10.5cm}{
\begin{center}
\textbf{ Algorithm {\bf (R-DFG)} }
\end{center}
{Given $x^{0,0} = y^{1,0} \in {\ca^*}$ and restart interval $K_c$. For $j
\geq 0$ do:}
\begin{enumerate}
\item Run Algorithm (DFG) for $K_c$ iterations to get ${x}^{K_c,j}$
\item Restart: $x^{0,j+1} = x^{K_c,j}$, \; $y^{1,j+1} = x^{K_c,j}$ \; and \;
$\theta_1=1$.
\end{enumerate}
}}
\end{center}
Then, after $p$ restarts of Algorithm {\bf (R-DFG)} we obtain linear
convergence in terms of dual suboptimality:

\begin{theorem}
\label{th_ebdualsuboptdfg} Under Assumption \eqref{as_strong} and
the error  bound property \eqref{eq_error_bound} for the
corresponding dual  of problem  \eqref{eq_prob_princc},  the
sequence $\left(x^{k,j},y^{k,j} \right)_{k,j\geq 0}$  generated by
Algorithm ({\bf R-DFG}) converges linearly in terms of the  dual
objective function values, i.e.:
\begin{equation}
\label{bound_dual_optim_dfgeb}
 f^* - d({x}^{K_c,p-1}) \leq \epsilon \quad \text{for} \quad k = p K_c =
 e  \kappa  \log \frac{L_\text{d} \mathcal{R}_\text{d}^2}{\epsilon} \quad \text{iterations}.
\end{equation}
\end{theorem}

\begin{proof}  After $p$ restarts of Algorithm {\bf (R-DFG)} we have:
\begin{align*}
f^* - d(x^{0,p})  & = f^* - d(x^{K_c,p-1})  \leq  \frac{2 L_\text{d} \|
x^{0,p-1} - \bar{x}^{0,p-1}\|^2}{(K_c+1)^2} \\
& \leq c^2 (f^* - d(x^{0,p-1})) \leq \cdots \leq c^{2p} (f^* - d(x^{0,0})).
\end{align*}
Thus, the total number of iterations is $p K_c$. Since $x^{0,0}=y^{1,0}$ it
follows that $x^{1,0}$ is the gradient step from $x^{0,0}$ and thus
$f^* - d(x^{1,0})  \leq \frac{L_\text{d}}{2} \|x^{0,0} -
\bar{x}^{0,0}\|^2$.  Therefore, we may assume for simplicity  that $f^* - d(x^{0,0})  \leq \frac{L_\text{d}}{2} \|x^{0,0} - \bar{x}^{0,0}\|^2$.  For $c = \frac{1}{e}$  we have:
\begin{align*}
& f^* - d(x^{K_c,p-1}) \leq c^{2p} (f^* - d(x^{0,0})) \leq  \frac{1}{e^{2p}} \frac{L_\text{d} \mathcal{R}_\text{d}^2}{2}  \leq    \epsilon,
\end{align*}
provided that we perform  $k= e  \kappa  \log
\frac{L_\text{d} \mathcal{R}_\text{d}^2}{\epsilon}$ number of iterations.
\end{proof}

\noindent Next theorem shows linear convergence in terms of primal suboptimality and infeasibility of the last primal iterate $v^k$ generated by Algorithm \textbf{(R-DFG)}.

\begin{theorem}
\label{th_urndfglast} Under the assumptions of Theorem \ref{th_ebdualsuboptdfg}, we get
an $\epsilon$-primal solution for \eqref{eq_prob_princc} in the last primal
iterate $v^{k} = u(x^{K_c,p-1})$ of Algorithm \textbf{(R-DFG)} after $k = p K_c =
\mathcal{O}(\log(\frac{1}{\epsilon}))$ iterations.
\end{theorem}

\begin{proof}
Combining \eqref{ineq_x} and \eqref{bound_dual_optim_dfgeb} we obtain the
following  relation:
\begin{align}
\label{dist_sublinebrdfg} \| v^k - u^* \| \leq  \frac{1}{e^p}  \sqrt{\frac{L_\text{d} \mathcal{R}_\text{d}^2}{\sigma_\text{f}}}.
\end{align}

\noindent Then, combining the previous relation
\eqref{dist_sublinebrdfg} and \eqref{ineq_feas2} we obtain a linear
estimate for feasibility violation of the last iterate $v^k$:
\begin{align}
\label{fes_sublinebdfg} \text{dist}_{\ca}(G v^k + g) & \leq   \|G\|   \| \bo{v}^k -
u^* \|  \leq  \frac{L_\text{d} \mathcal{R}_\text{d}}{e^p},
\end{align}
where we used the definition of  $L_\text{d} =
\|G\|^2/\sigma_\text{f}$. Finally, we derive linear estimates for
primal suboptimality of the last iterate $u^k$. Combining
\eqref{dist_sublinebrdfg} with \eqref{dfg_pr1}  we get:
\begin{align}
\label{opt_sublinebdfg}
|f(v^k) - f^*| & \leq \norm{G}(2\mathcal{R}_{\text{d}} + \norm{x^0}) \| v^k - u^* \| \leq  \frac{\norm{G}(2\mathcal{R}_{\text{d}} + \norm{x^0})}{e^p}  \sqrt{\frac{L_\text{d} \mathcal{R}_\text{d}^2}{\sigma_\text{f}}}  \nonumber \\
& = \frac{L_\text{d} \mathcal{R}_\text{d} (2\mathcal{R}_{\text{d}} + \norm{x^0})}{e^p} .
\end{align}
In conclusion, we get an
$\epsilon$-primal solution in the last primal iterate $v^{k}$ provided that we perform $k = p K_c = e  \kappa  \log \frac{L_\text{d} \mathcal{R}_\text{d}^2}{\epsilon}$  iterations of Algorithm \textbf{(R-DFG)}.
\end{proof}

\noindent From our knowledge, the results stated in  Theorems
\ref{th_ebdualsuboptdfg} and \ref{th_urndfglast} answer for the
first time to a question posed by Tseng \cite{Tse:10} related to
whether there exist fast gradient schemes  that converge linearly
on convex problems having an error bound property.


\section{Better convergence rates  for dual
first order methods in the last primal iterate for
linearly constrained convex problems}
\label{extensions}

In this section we prove that for linearly constrained convex
problems ($\ca = \{0\}$) we can get better iteration complexity estimates for dual first order methods  corresponding to the last primal iterate sequence.  More precisely, we prove
that we can improve substantially the convergence rate of  dual
first order methods (\textbf{DG}) and (\textbf{DFG}))  in the last
iterate when the optimization problem \eqref{eq_prob_princc} has
linear equality constraints:  i.e.  $Gu + g =0$ instead of $Gu +g \in \ca$. Therefore, in this section we consider a particular case for the  optimization
problem \eqref{eq_prob_princc}, namely  a linearly constrained
convex optimization problem of the form:
\begin{align}
\label{ling} \min_{u \in U} f(u): \quad \text{s.t.} \quad  Gu + g
=0.
\end{align}
For \eqref{ling} we still require Assumption \eqref{as_strong} to hold: i.e. $f$ is $\sigma_\text{f}$-strongly convex function, $U$
a simple convex set and there exists a finite optimal Lagrange multiplier $x^*$.
Since $f$ is strongly convex and $\ca = \{0\}$, the dual gradient $\nabla d(x) =  G u(x) + g$ is Lipschitz continuous with constant $L_\text{d} =
\frac{\|G\|^2}{\sigma_\text{f}}$  (see e.g. \cite{Nes:05}). We
analyze below the convergence behavior of dual first order methods
for solving the linearly constrained  convex problem
\eqref{ling}. Note that since we have linear constraints in
\eqref{ling}, i.e.  ${\ca} = \{0\}$, the  corresponding  dual problem is unconstrained, i.e. ${\ca^*} = \rset^p$.

\vspace{0.2cm}

\noindent \textbf{Case 1}:  We first consider  applying  $2k$ steps of  Algorithm  (\textbf{DG}).  For simplicity,   let us assume constant step size $\alpha_k =
1/L_\text{d}$ for solving the corresponding  dual of  problem
\eqref{ling}.

\begin{theorem}
\label{th_lingdglast} For problem \eqref{ling} let $f$  be
$\sigma_\text{f}$-strongly convex function, $U$   be simple convex set
and the set of optimal multipliers $X^*$ be nonempty. Further,  let the
sequences $\left(x^k,u^k\right)_{k\geq 0}$ be generated by Algorithm
{\bf (DG)} with $\alpha_k = 1/L_\text{d}$. Then, for a given
accuracy $\epsilon>0$ we get an $\epsilon$-primal solution for
\eqref{ling} in the last primal  iterate $u^{2k}$ of Algorithm
\textbf{(DG)} after $2k = {\mathcal O} (\frac{1}{\epsilon})$
iterations.
\end{theorem}

\begin{proof}
We have proved in \eqref{dg_pr2} that gradient algorithm is an ascent method, i.e.:
\begin{align*}
d(x^{j+1}) - d(x^j)  \geq  \frac{L_\text{d}}{2}  \|x^{j+1} - x^j \|^2 =  \frac{L_\text{d}}{2} \|\nabla^+ d(x^j)\|^2   \qquad \forall j \geq 0.
\end{align*}

\noindent  Adding for $j=k$ to $j=2k$ and using that the gradient map sequence is decreasing along the iterations of Algorithm {\bf (DG)} (see Lemma \ref{lemma2_dg}),  we get:
\begin{align}
\label{sumk2k} d(x^{2k+1}) -  d(x^{k})  & \geq  \sum_{j=k}^{2k}
\frac{L_\text{d}}{2} \|\nabla^+ d(x^j)\|^2 \overset{\eqref{decrease_gm}}{\geq} \frac{L_\text{d}(k+1)}{2} \|\nabla^+ d(x^{2k})\|^2  \nonumber \\
& \overset{\eqref{decrease_ggm}}{\geq} \frac{k+1}{2L_\text{d}} \|\nabla d(x^{2k})\|^2.
\end{align}
Since $d(x^{2k+1}) \leq f^*$, we obtain:
\[ \frac{k+1}{2L_\text{d}} \|\nabla d(x^{2k})\|^2  \overset{\eqref{sumk2k}}{\leq} f^* -
d(x^{k}) \overset{\eqref{bound_dual_optim_adg}}{\leq} \frac{4
L_\text{d} \mathcal{R}_\text{d}^2}{k}.   \] From $\nabla d(x) = G
u(x) + g$  we obtain a sublinear estimate for feasibility violation
of the last primal iterate $u^{2k} = u(x^{2k})$ of Algorithm
(\textbf{DG}):
\begin{align}
\label{lingpf} \|G u^{2k} + g \| = \|\nabla d(x^{2k})\|   \leq  \frac{3 L_\text{d} \mathcal{R}_\text{d}}{k}.
\end{align}

\noindent We can also characterize primal suboptimality in the last
iterate $u^{2k} $ for Algorithm (\textbf{DG}) using that $G
u^* + g=0$, the  estimate on infeasibility  \eqref{lingpf}  and the
inequalities \eqref{ineq_optl}--\eqref{ineq_feas3}:
\begin{align}
\label{lingps} |f(u^{2k}) - f^*| & \leq   \left( \|x^{2k} - x^*\| +
\|x^*\|
\right ) \|G u^{2k} + g \|  \overset{\eqref{dg_pr1} + \eqref{lingpf}}{\leq}  (2
\mathcal{R}_\text{d} + \|x^0\|) \frac{3 L_\text{d}
\mathcal{R}_\text{d}}{k}.
\end{align}

\noindent Therefore, we have obtained sublinear estimates  of order
$\mathcal{O}(\frac{1}{k})$ for primal infeasibility (inequality
\eqref{lingpf}) and primal suboptimality (inequality \eqref{lingps})
for the last primal iterate sequence $(u^k)_{k \geq 0}$ generated by
Algorithm~\textbf{(DG)}. Now, it is  straightforward to see that if
we want to get an $\epsilon$-primal solution in $u^{2k}$ we need to
perform $2k = {\mathcal O} (\frac{1}{\epsilon})$ iterations.
\end{proof}

\noindent  In conclusion,  from Theorem \ref{th_lingdglast} it
follows that we obtain $\epsilon$-accuracy for primal suboptimality
and infeasibility for Algorithm (\textbf{DG}) in the last primal
iterate $u^{k}$ after $k = \mathcal{O}(\frac{1}{\epsilon})$
iterations. This is better than $\mathcal{O}(\frac{1}{\epsilon^2})$
iterations  obtained in Theorem \ref{th_dglast}  for the last primal
iterate $u^k$ and it is of the same order as for the primal average
sequence  $\hat u^k$ from Theorem \ref{th_dgav}. However, this
better result is obtained for the particular linearly constrained
convex problem \eqref{ling}. Note that an immediate consequence of
Lemma \ref{lemma2_dg}  for this case $\ca^* = \rset^p$  is that the
sequence  $\|\nabla d(x^j)\|$ is  decreasing, i.e.: $$\|\nabla
d(x^{j+1})\| \leq \|\nabla d(x^{j})\|.  \quad  \forall j \geq 0$$

\vspace{0.2cm}

\noindent \textbf{Case 2}:  We now consider an hybrid algorithm
that applies  $k$ steps of  Algorithm  (\textbf{DFG}) and then $k$
steps of  Algorithm  (\textbf{DG}) for solving the corresponding
dual of  problem \eqref{ling}.

\begin{theorem}
\label{th_linghdfgdglast} Under the assumptions of Theorem
\ref{th_lingdglast}  let the sequences $\left(x^k, y^k,
u^k\right)_{k\geq 0}$ be generated by  applying $k$ steps of
Algorithm  (\textbf{DFG}) and then $k$ steps of  Algorithm
(\textbf{DG})  with $\alpha_k = 1/L_\text{d}$. Then, for a given
accuracy $\epsilon>0$ we get an $\epsilon$-primal solution for
\eqref{ling} in the last primal  iterate $u^{2k}$ of  this algorithm
after $2k = {\mathcal O} (\frac{1}{\epsilon^{2/3}})$ iterations.
\end{theorem}

\begin{proof}
Since  the gradient algorithm is an ascent method (see \eqref{dg_pr2}), we have:
\begin{align*}
d(x^{j+1}) - d(x^j)  \geq  \frac{L_\text{d}}{2}  \|x^{j+1} - x^j \|^2 =  \frac{L_\text{d}}{2} \|\nabla^+ d(x^j)\|^2   \qquad \forall j \geq k.
\end{align*}

\noindent  Adding for $j=k$ to $j=2k$ and using the decrease of the  gradient map,  we get:
\begin{align*}
 d(x^{2k+1})  -  d(x^{k})  & \geq  \sum_{j=k}^{2k}
\frac{L_\text{d}}{2} \|\nabla^+ d(x^j)\|^2
\overset{\eqref{decrease_gm}}{\geq}  \frac{L_\text{d}(k+1)}{2}
\|\nabla^+ d(x^{2k})\|^2   \overset{\eqref{decrease_ggm}}{\geq}
\frac{k+1}{2L_\text{d}} \|\nabla d(x^{2k})\|^2.
\end{align*}

\noindent Since $d(x^{2k+1}) \leq f^*$, we obtain:
$\frac{k+1}{2L_\text{d}}  \|\nabla d(x^{2k})\|^2  \leq f^* -
d(x^{k}) \overset{\eqref{bound_dual_optim_dfg}}{\leq} \frac{2
L_\text{d} \mathcal{R}_\text{d}^2}{(k+1)^2}$.   From $\nabla d(x) =
G u(x) + g$  we obtain a sublinear estimate for feasibility
violation of the last primal iterate  $u^{2k} = u(x^{2k})$ of this
hybrid algorithm:
\begin{align}
\label{lingpfh} \|G u^{2k} + g \| = \|\nabla d(x^{2k})\|   \leq
\frac{2 L_\text{d} \mathcal{R}_\text{d}}{(k+1)^{3/2}}.
\end{align}

\noindent We can also characterize primal suboptimality in the last
iterate $u^{2k} $ for  this hybrid algorithm  using that $G
u^* + g=0$, the  estimate   \eqref{lingpfh}  and the
inequalities \eqref{ineq_optl}--\eqref{ineq_feas3}:
\begin{align}
\label{lingpsh} |f(u^{2k}) - f^*| & \!\leq\!   \left( \|x^{2k} -
x^*\| + \|x^*\|\right ) \|G u^{2k} + g \| \!\overset{\eqref{lingpfh}
+ \eqref{dfg_pr1}}{\leq}\! (2\mathcal{R}_\text{d} + \|x^0\|) \frac{2
L_\text{d} \mathcal{R}_\text{d}}{(k+1)^{3/2}}.
\end{align}

\noindent Therefore, we have obtained sublinear estimates  of order
$\mathcal{O}(\frac{1}{k^{3/2}})$ for primal infeasibility
(inequality \eqref{lingpfh}) and primal suboptimality (inequality
\eqref{lingpsh}) for the last primal iterate sequence $(u^k)_{k \geq
0}$ generated by an algorithm applying $k$ steps  of \textbf{(DFG)}
and then $k$ steps  of \textbf{(DG)}. Now, it is straightforward to
see that if we want to get an $\epsilon$-primal solution in $u^{2k}$
we need to perform $2k = {\mathcal O} (\frac{1}{\epsilon^{2/3}})$
iterations.
\end{proof}

\noindent For the linear constrained problem \eqref{ling} in
\cite{BecTeb:14} convergence rate $\mathcal{O}(\frac{1}{k})$ was
derived for the last primal iterate of Algorithm  \textbf{(DFG)}
(see also our Theorem \ref{th_dfglast} that  gives the same
convergence rate for conic problems). However, Theorem
\ref{th_linghdfgdglast} shows that applying further $k$ gradient
steps we can improve the convergence rate to
$\mathcal{O}(\frac{1}{k^{3/2}})$ for problem \eqref{ling}.

\vspace{0.2cm}

\noindent  In conclusion, in this paper we obtained the following
estimates for the convergence rate of dual first order methods:
\begin{itemize}
\item in a primal average sequence we have
$\mathcal{O}(\frac{1}{\epsilon})$ for Algorithm (\textbf{DG}) and
$\mathcal{O}(\sqrt{\frac{1}{\epsilon}})$ for Algorithm
(\textbf{DFG})
\item  in the last iterate they are are summarized in Table 1.
\end{itemize}

\begin{table}[ht]
\centering \caption{Rate of convergence estimates of dual first
order methods in the last primal iterate.}

\begin{tabular} {| c || c | c | p{2cm} | c | c | c | p{2cm} |}
\hline
Alg.  &  DG     &  DFG  &  regularized  DFG   &  DG  &  R-DFG  &  $2k$-DG & hybrid  DFG-DG  \\
[1ex]\hline Prob.   & \eqref{eq_prob_princc} &
\eqref{eq_prob_princc} & \eqref{eq_prob_princc} &
\eqref{eq_prob_princc}$+$\eqref{eq_error_bound} &
  \eqref{eq_prob_princc}$+$\eqref{eq_error_bound} & \eqref{ling}  & \eqref{ling}  \\
[1ex]\hline Rates & $\mathcal{O}(\frac{1}{\epsilon^2})$ &
$\mathcal{O}(\frac{1}{\epsilon})$ &
$\mathcal{O}(\sqrt{\frac{1}{\epsilon}} \log \frac{1}{\epsilon})$  &
$\mathcal{O}(\log \frac{1}{\epsilon})$ & $\mathcal{O}(\log
\frac{1}{\epsilon})$ & $\mathcal{O}(\frac{1}{\epsilon})$ &
$\mathcal{O}(\frac{1}{\epsilon^{2/3}})$   \\
[1ex] \hline
\end{tabular}
\end{table}


\section{Better convergence rates  for dual first order methods
in the last primal iterate for conic convex problems}

In this section we prove that some of the results of the previous
section can be extended to conic  convex  problem
\eqref{eq_prob_princc}. More precisely, we prove that we can improve
substantially the convergence estimates for primal infeasibility and
left hand side suboptimality  of dual first order methods  in the last iterate for the general problem \eqref{eq_prob_princc}.

\vspace{0.2cm}

\noindent \textbf{Case 1}:  We first consider  applying $2k$ steps
of  Algorithm  (\textbf{DG}).  For simplicity,    let us assume
constant step size $\alpha_k = 1/L_\text{d}$ for solving the
corresponding  dual of  problem \eqref{eq_prob_princc}. Indeed, we
have proved in \eqref{dg_pr2} that gradient algorithm is an ascent
method, i.e.:
\begin{align*}
d(x^{j+1}) - d(x^j)  \geq  \frac{L_\text{d}}{2}  \|x^{j+1} - x^j \|^2 =
\frac{L_\text{d}}{2} \|\nabla^+ d(x^j)\|^2   \qquad \forall j \geq 0.
\end{align*}

\noindent  Adding for $j=k$ to $j=2k$ and using that the gradient
map sequence is decreasing along the iterations of Algorithm {\bf
(DG)} (see Lemma \ref{lemma2_dg}),  we get:
\begin{align}
\label{sumk2k_cone} d(x^{2k+1}) -  d(x^{k})  & \geq  \sum_{j=k}^{2k}
\frac{L_\text{d}}{2} \|\nabla^+ d(x^j)\|^2
\overset{\eqref{decrease_gm}}{\geq} \frac{L_\text{d}(k+1)}{2}
\|\nabla^+ d(x^{2k})\|^2  \nonumber \\
& \overset{\eqref{decrease_ggm}}{\geq} \frac{k+1}{2L_\text{d}}
 d_{\ca} \left(-\nabla d(x^{2k}) \right)^2.
\end{align}
Since $d(x^{2k+1}) \leq f^*$, we obtain:
\[ \frac{k+1}{2L_\text{d}} d_{\ca} \left(-\nabla d(x^{2k}) \right)^2  \overset{\eqref{sumk2k}}{\leq} f^* -
d(x^{k}) \overset{\eqref{bound_dual_optim_adg}}{\leq} \frac{4
L_\text{d} \mathcal{R}_\text{d}^2}{k}.   \] From $\nabla d(x) = -G
u(x) - g$  we obtain a sublinear estimate for feasibility violation
of the last primal iterate $u^{2k} = u(x^{2k})$ of Algorithm
(\textbf{DG}):
\begin{align}
\label{conegpf} d_{\ca}(G u^{2k} + g) = d_{\ca}(-\nabla d(x^{2k}))
\leq  \frac{3 L_\text{d} \mathcal{R}_\text{d}}{k}.
\end{align}

\noindent We can also characterize primal suboptimality in the last
iterate $u^{2k} $ for Algorithm (\textbf{DG}). On one hand, using
the  estimate on infeasibility  \eqref{conegpf} and the definition
of the dual cone $\ca^*$, we have:
\begin{align}\label{conegps_left}
f^* &=  \min_{u \in U} f(u) + \langle x^*, g(u) \rangle \leq f(u^{2k})
+ \langle x^*, g(u^{2k}) \rangle \nonumber\\
& =  f(u^{2k}) + \langle x^*, -G u^{2k} - g \rangle \leq f(u^{2k})
+ \langle x^*,  [G u^{2k} + g]_{\ca} - (G u^{2k} + g) \rangle \nonumber\\
& \leq f(u^{2k}) + \norm{x^*} \text{dist}_{\mathcal{K}}(G u^{2k} +
g)  \overset{\eqref{conegps_left}}{\leq} f(u^{2k}) + \frac{3
L_{\text{d}} \mathcal{R}_\text{d}}{k} ( \mathcal{R}_\text{d} +
\|x^0\|).
\end{align}
On the other hand, using \eqref{ineq_optr}, we have
\begin{align}
\label{conegps_right} f(u^{2k}) - f^* & \leq   \left( \|x^{2k} - x^*\| +
\|x^*\| \right )\|G\| \|u^{2k} - u^* \|  \nonumber \\
&\overset{\eqref{dg_pr1} + \eqref{dist_sublin_dg}}{\leq}  2
L_\text{d} \mathcal{R}_\text{d} (2 \mathcal{R}_\text{d} + \|x^0\|)
\sqrt{\frac{1}{k}}.
\end{align}

\noindent Therefore, we have obtained sublinear estimates  of order
$\mathcal{O}(\frac{1}{k})$ for primal infeasibility (inequality
\eqref{conegpf}) and left hand side suboptimality (inequality
\eqref{conegps_left}) and of order $\mathcal{O}(\frac{1}{\sqrt{k}})$
for right hand side primal suboptimality (inequality
\eqref{conegps_right}) for the last primal iterate sequence
$(u^k)_{k \geq 0}$ generated by Algorithm~\textbf{(DG)}.

\vspace{0.2cm}

\noindent \textbf{Case 2}:  We now consider an hybrid  algorithm
that applies  $k$ steps of  Algorithm (\textbf{DFG}) and then $k$
steps of  Algorithm  (\textbf{DG}) for solving the corresponding
dual of problem \eqref{eq_prob_princc}. Since  the gradient
algorithm is an ascent method (see \eqref{dg_pr2}), we have:
\begin{align*}
d(x^{j+1}) - d(x^j)  \geq  \frac{L_\text{d}}{2}  \|x^{j+1} - x^j
\|^2 =   \frac{L_\text{d}}{2} \|\nabla^+ d(x^j)\|^2   \qquad \forall
j \geq k.
\end{align*}

\noindent  Adding for $j=k$ to $j=2k$ and using the decrease of the  gradient map,  we get:
\begin{align*}
 d(x^{2k+1})  -  d(x^{k})  \geq  \sum_{j=k}^{2k}
\frac{L_\text{d}}{2} \|\nabla^+ d(x^j)\|^2 & \overset{\eqref{decrease_gm}}{\geq}
\frac{L_\text{d}(k+1)}{2} \|\nabla^+ d(x^{2k})\|^2 \\
& \overset{\eqref{decrease_ggm}}{\geq} \frac{k+1}{2L_\text{d}}
 d_{\ca}(-\nabla d(x^{2k}))^2.
\end{align*}

\noindent Since $d(x^{2k+1}) \leq f^*$, we obtain:
$\frac{k+1}{2L_\text{d}} d_{\ca}(-\nabla d(x^{2k}))^2  \leq f^* -
d(x^{k}) \overset{\eqref{bound_dual_optim_dfg}}{\leq} \frac{2
L_\text{d} \mathcal{R}_\text{d}^2}{(k+1)^2}$.   From $\nabla d(x) =
- G u(x) - g$  we obtain a sublinear estimate for feasibility
violation of the last primal iterate  $u^{2k} = u(x^{2k})$ of this
hybrid algorithm:
\begin{align}
\label{gpfh_cone} d_{\ca}(G u^{2k} + g) = d_{\ca}(-\nabla d(x^{2k}))
\leq  \frac{2 L_\text{d} \mathcal{R}_\text{d}}{(k+1)^{3/2}}.
\end{align}

\noindent We can also characterize primal suboptimality in the last
iterate $u^{2k} $ for  this hybrid algorithm. Using the  estimate
\eqref{gpfh_cone}, a similar reasoning as in the relations
\eqref{conegps_left} leads to:
\begin{align}\label{fgpsh_cone_left}
 -\frac{2L_{\text{d}}(\mathcal{R}_{\text{d}}^2 + \mathcal{R}_{\text{d}} \norm{x^0})}{(k+1)^{3/2}} \le
 -\frac{2L_{\text{d}}\mathcal{R}_{\text{d}} \norm{x^*}}{(k+1)^{3/2}} \le f(u^{2k})  - f^*.
\end{align}
On the other hand, from \eqref{ineq_optr} it can be derived:
\begin{align} \label{fgpsh_cone_right}
f(u^{2k}) - f^* & \leq   \left( \|x^{2k} - x^*\| +
\|x^*\| \right )\|G\| \|u^{2k} - u^* \|  \nonumber \\
& \leq  (2 \mathcal{R}_\text{d} + \|x^0\|) \frac{3 L_\text{d}
\mathcal{R}_\text{d}}{k+1}.
\end{align}

\noindent Therefore, we have obtained sublinear estimates  of order
$\mathcal{O}(\frac{1}{k^{3/2}})$ for primal infeasibility
(inequality \eqref{gpfh_cone}) and for left hand side primal
suboptimality (inequality \eqref{fgpsh_cone_left}) and of order
$\mathcal{O}(\frac{1}{k})$ for right hand side primal suboptimality
(inequality \eqref{fgpsh_cone_right}) for the last primal iterate
sequence $(u^k)_{k \geq 0}$ generated by Algorithm~\textbf{(DG)}.

\section{Numerical simulations}
\label{sec_numerical} For  numerical experiments we consider random
problems of the following form:

\begin{align*}
& \min_{u \in \rset^n} \frac{1}{2} u^T Q u + q^T u + \gamma \log(1 + a^T x + e^{b^T u}) \\
& \text{s.t.}: \quad Gu + g \leq 0, \quad \text{lb} \leq u \leq
\text{ub},
\end{align*}
where $Q$ is positive definite matrix with
$\sigma_\text{f}=\lambda_\text{min}(Q) = 1$, $G \in \rset^{3n/2
\times n}$, $q, a, b \in \rset^n$  and $\gamma \in \rset$. We need
to remark that the objective function is not convex for $a, b \not =
0$, but it is convex e.g. when $(\gamma < 0, a \geq 0, b=0)$ on
$\rset^n_+$ or when $(\gamma  > 0, a = 0, b \not =0)$ on $\rset^n$.
Note that this type of problems arises in many practical
applications: in network utility maximization \cite{BecNed:14}
$(\gamma <0, a \geq 0, b=0)$; in resource allocation problems
\cite{XiaBoy:06} $(\gamma >0, a=0, b \not = 0)$; in optimal power
flow or model predictive control \cite{NecNed:13} ($\gamma=0$). All
the data of the problem are generated randomly and $G$ is sparse
having  tens of nonzeros ($\simeq 50$) on each row for large
problems ($n \gg 10^3$). We have considered the accuracy
$\epsilon=10^{-2}$, the value for $\gamma= \pm 0.5$  and the
stopping criteria in the tables below were chosen as follows:
\[  \text{ds}= | d(x^{k+1})  - d(x^k)| \leq \epsilon^2
\quad \text{and} \quad \text{pf}= \|[G w^k + g]_+\| \leq \epsilon,
\] where $w^k$ is either the last primal iterate ($u^k/v^k$) or average
of primal iterates~($\hat{u}^k$) and we allow at most $15000$ number
of iterations for each algorithm.

\subsection{Case 1: $(\gamma < 0, a > 0, b=0)$}
In the first set of experiments we choose $(\gamma < 0, a > 0, b=0)$
and  simple constraints $u \geq 0$ (e.g. network utility
maximization problems \cite{BecNed:14} can be recast in this form).
In this type of applications the complicating constraints $Gu + g
\leq 0$ are related to the capacity of the links and we need to also
impose   simple  constraints $u \geq 0$, since  $u$ represents  the
source rates. Note that the objective function is  strongly convex
and with  Lipschitz gradient on $U=\rset^n_+$.  However, the presence of simple
constraints $u \geq 0$ makes the dual function degenerate (i.e. $d$ does not
satisfy  an error bound property).

\noindent Typically, the performance  in  terms of primal
suboptimality and infeasibility of Algorithms \textbf{(DG)} and
\textbf{(DFG)} in the primal last iterate or in the  average of
primal iterates   is oscillating as Fig. 1 shows. However, these
algorithms have a smoother behavior  in the average of iterates than
in the last iterate.  Moreover, from our numerical experience we
have observed that for our dual first order methods we usually have
a better behavior in the last iterate than in the average of
iterates as we can also see from Fig. 1 and Table 1 (in the table we
display the average number of iterations for $10$ random problems
for each dimension $n$ ranging from $10$ to $10^4$). On the other
hand,  our  worst case convergence analysis says differently, i.e.
we have obtained better theoretical estimates in the primal  average
sequence than in the last primal  iterate sequence. This does not
mean that our  analysis is weak, since we can also construct
problems which show the behavior predicted by our theory, see e.g.
Fig. 2 where indeed we have a better behavior in the average of
iterates than in the last iterate.

\noindent Finally, in Fig. 3 we plot the practical  number of
iterations of Algorithms \textbf{(DG)} and \textbf{(DFG)} for
different test cases of the same dimension $n=50$ (left) and for
different test cases of variable dimension ranging from $n=10$ to
$n=500$ (right). From this figure we observe that the number of
iterations are not varying much for different test cases and also
that the number of iterations are mildly dependent of problem's
dimension.

\begin{figure}[ht]
\label{fig_contradict}
\caption{Typical performance  in terms of
primal suboptimality and infeasibility of
Algorithms \textbf{(DG)} in the last iterate (DG-last), \textbf{(DG)} in average (DG-average),
\textbf{(DFG)} in the last iterate (DFG-last) and \textbf{(DFG)} in average (DFG-average) for $n=50$.}
\centerline{\includegraphics[height=5.5cm,width=13cm]{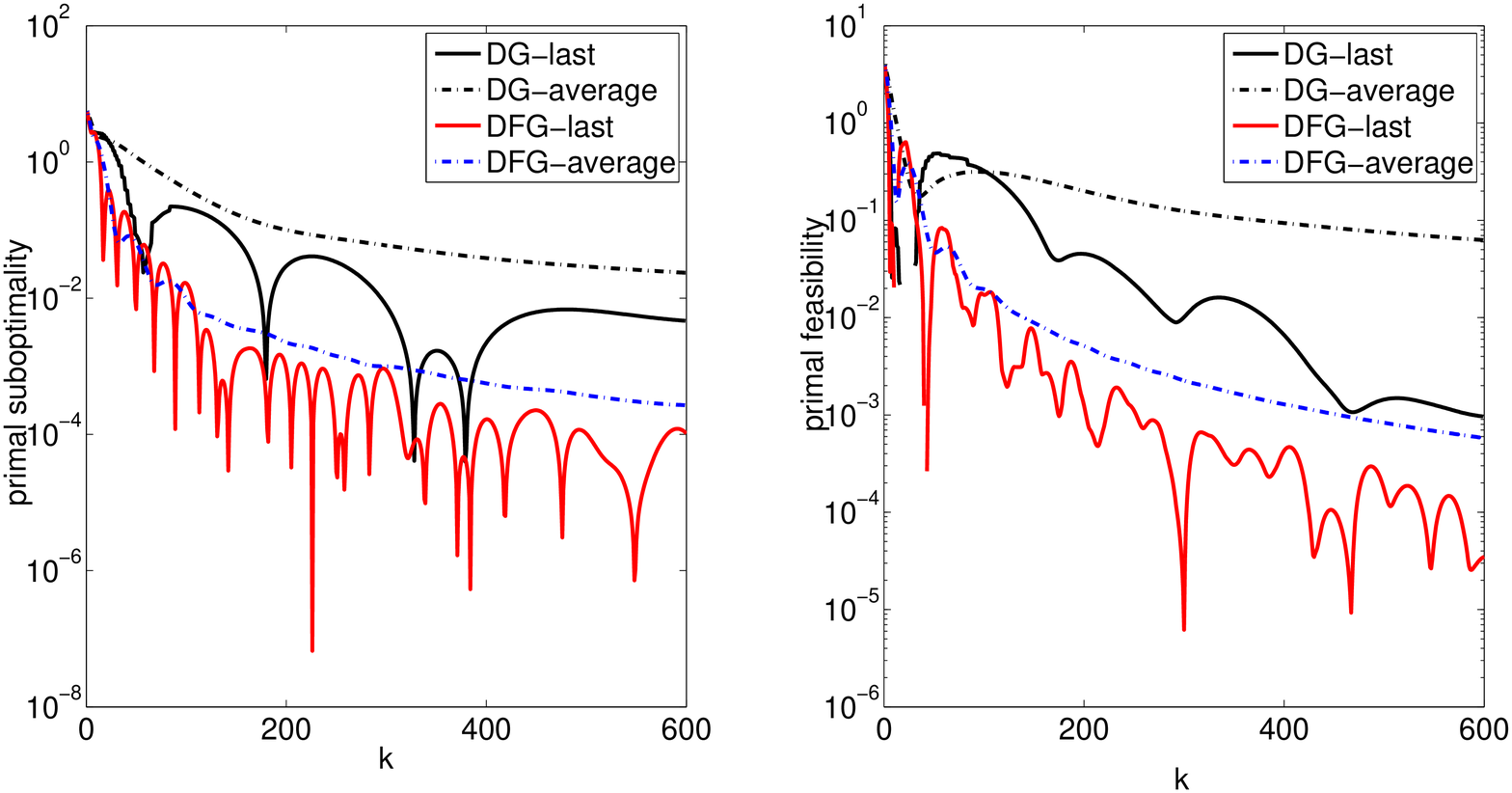}}
\end{figure}

\begin{table*}[ht]
\label{table_nss} \centering \caption{Average number of iterations
for $10$ random problems for each dimension $n$ for Algorithms
\textbf{(DG)} and \textbf{(DFG)} in the last iterate and in the
average of iterates.   We  observe that dual first order methods perform
better in the primal last iterate than in the average of iterates.}
\begin{tabular}{|c|c|c|c|c|c|c|}
\hline Alg./n                     & $10$  & $50$  & $10^2$  &  $10^3$   & $5*10^3$  &  $10^4$  \\
\hline
$k_{\mathrm{last}}^{\mathrm{DG}}$ & $44$  & $519$ & $621$ & $5546$  &  $7932$   &  $9207$\\
\hline

$k_{\mathrm{avg.}}^\mathrm{DG}$  & $504$ & $1498$ & $3706$ & $9830$ & $-$ & $-$ \\
\hline \hline

$k_{\mathrm{last}}^{\mathrm{DFG}}$   & $13$   & $75$ & $92$  & $382$    & $691$    & $1145$ \\
\hline

$k_{\mathrm{avg.}}^{\mathrm{DFG}}$   & $28$    & $88$ & $123$  & $602$  & $1078$   & $1981$ \\
\hline
\end{tabular}
\end{table*}

\begin{figure}[ht]
\label{fig_fit} \caption{Practical performance comparable with the
theoretical estimates for primal suboptimality and infeasibility of
Algorithms  \textbf{(DG)} in the last iterate (DG-last),
\textbf{(DG)} in average (DG-average), \textbf{(DFG)} in the last
iterate (DFG-last) and \textbf{(DFG)} in average (DFG-average) for
$n=100$.}
\centerline{\includegraphics[height=5.5cm,width=13cm]{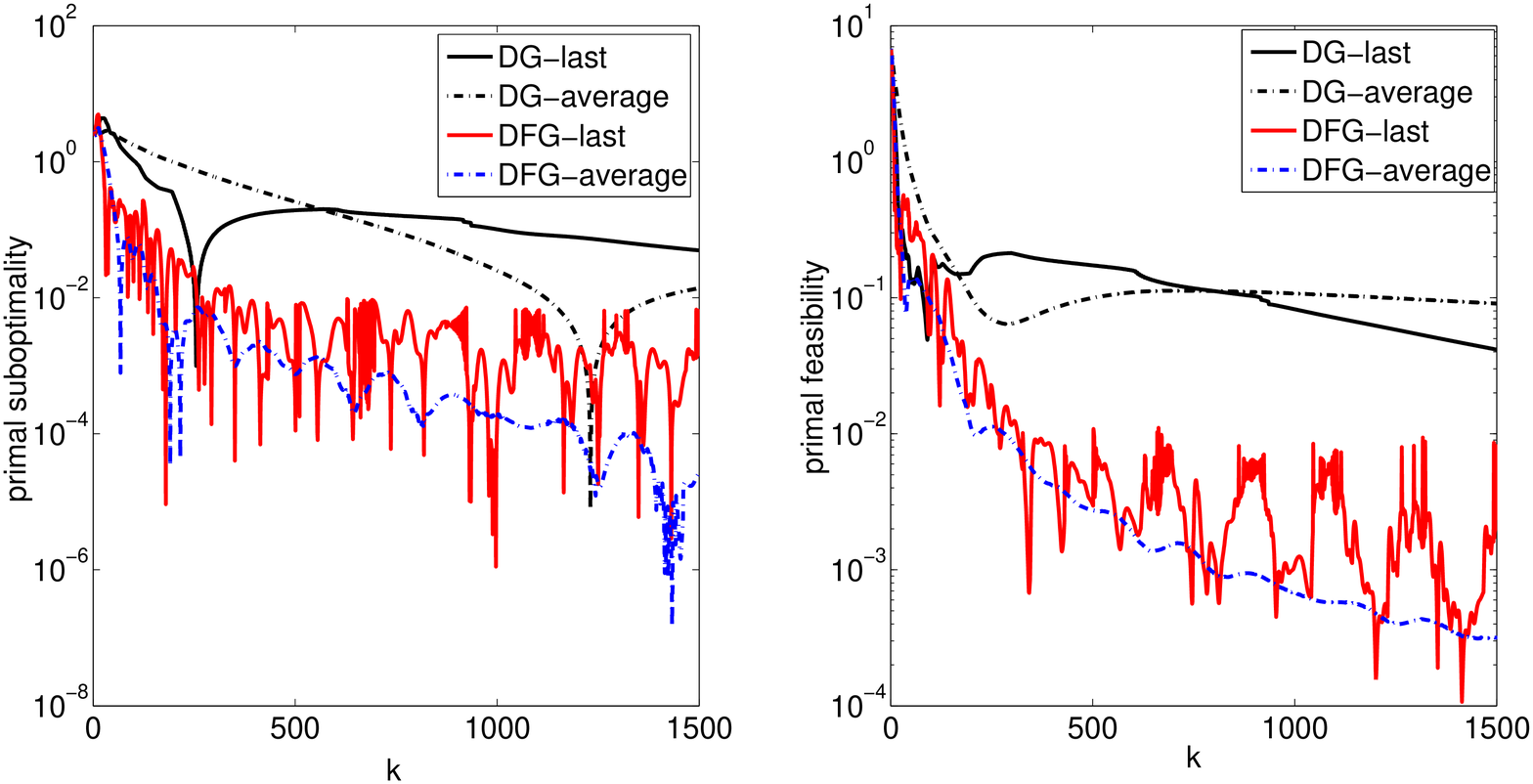}}
\end{figure}

\begin{figure}[ht]
\label{fig_fitt} \caption{Practical number of iterations  of
Algorithms \textbf{(DG)} in the last iterate (DG-last),
\textbf{(DG)} in average (DG-average), \textbf{(DFG)} in the last
iterate (DFG-last) and \textbf{(DFG)} in average (DFG-average) for
$30$ random test cases of fixed dimension  $n=50$ (left) or variable
dimension ranging from $n=10$ to $n = 500$ (right).}
\centerline{\includegraphics[height=5.5cm,width=13cm]{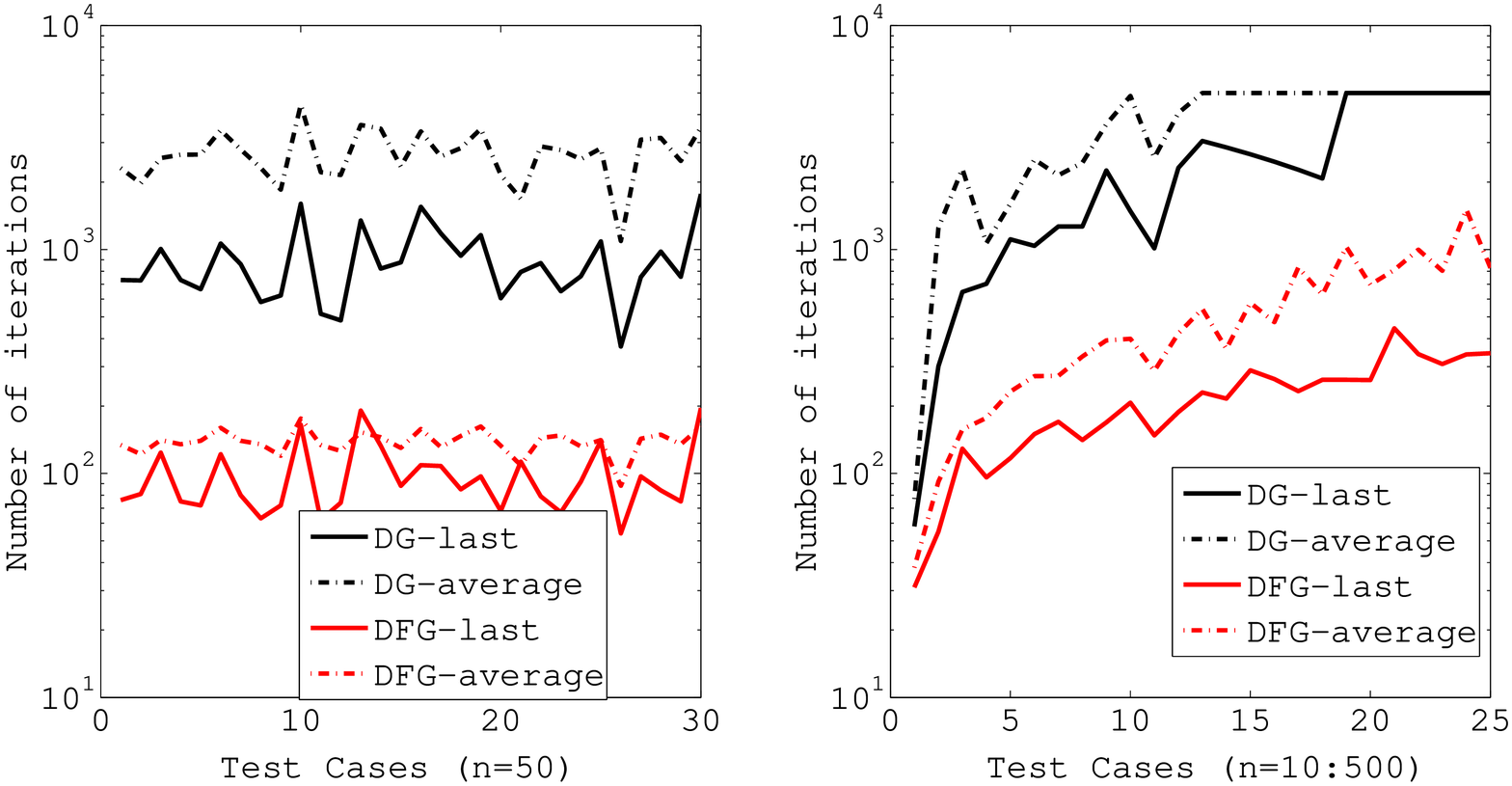}}
\end{figure}


\subsection{Case 2: $\gamma>0, a  = 0, b \not =0$}
\noindent In the second set of experiments we choose $(\gamma>0, a
= 0,  b \not =0)$  and simple box constraints $\text{lb} \leq u \leq
\text{ub}$ defining the set $U$ (e.g. this optimization model, in separable form, was
considered  in \cite{XiaBoy:06} for resource allocation problems).
In this case the objective function is strongly convex and has
Lipschitz gradient on $\rset^n$. Therefore, if the simple box constraints   are
missing,  then according to our theory given in Section
\ref{sec_lin} Algorithm \textbf{(DG)} is converging linearly.

\noindent We first  consider  box constraints $U=[\text{lb} \; \text{ub}]$
and the results (average number of iterations) are shown in Table 2
for $10$ random problems for each dimension $n$ ranging from $10$ to
$10^4$. We can again observe that dual first order methods perform
better in the primal last iterate than in the average of iterates. Further, we can notice
that the behavior of Algorithm \textbf{(DG)} in the last iterate is
comparable to that of Algorithm \textbf{(DFG)} in average. However,
the inner problem has to be solved with higher accuracy in Algorithm
\textbf{(DFG)} than in \textbf{(DG)} since the first one is more
sensitive to errors, such as inexact first order information, than
the last one (see \cite{NecNed:13} for a more in depth discussion on inexact dual
first order methods).
\begin{table*}[ht]
\label{table_ns} \centering \caption{Average number of iterations
for $10$ random problems for each dimension $n$ for Algorithms
\textbf{(DG)} and \textbf{(DFG)} in the last iterate and in the
average of iterates. We can again observe that dual first order methods perform
better in the primal last iterate than in the average of iterates.}
\begin{tabular}{|c|c|c|c|c|c|c|}
\hline Alg./n                     & $10$  & $50$  & $10^2$  &  $10^3$   & $5*10^3$  &  $10^4$  \\
\hline
$k_{\mathrm{last}}^{\mathrm{DG}}$ & $35$  & $195$ & $463$ & $782$  &  $1147$   &  $2155$\\
\hline

$k_{\mathrm{avg.}}^\mathrm{DG}$  & $527$ & $3423$ & $12697$ & $-$ & $-$ & $-$ \\
\hline \hline

$k_{\mathrm{last}}^{\mathrm{DFG}}$  & $19$   & $61$ & $97$  & $198$    & $276$    & $292$ \\
\hline

$k_{\mathrm{avg.}}^{\mathrm{DFG}}$ & $41$    & $108$ & $186$  & $381$  & $563$   & $582$ \\
\hline
\end{tabular}
\end{table*}

\noindent Then, we also take $\gamma=0$ and we solve the corresponding QP problems   over an
increasing dimension $n=10$ to $10^3$. In Fig. 4  we compare for Algorithm \textbf{(DFG)}  the real number of iterates in the primal latest iterate and average of iterates   and the estimated number of iterates ${\cal O}(1/k^2)$  for a primal suboptimality and infeasibility level of $\varepsilon=10^{-2}$. We observe from Fig. 4
that our theoretical estimates are quite close to the practical ones for the
dual fast gradient method.
\begin{figure}[ht]
\label{real_theory}
\caption{Real number of iterates  in the primal latest iterate and average of iterates
 and the estimated number of iterates ${\cal O}(1/k^2)$ for Algorithm \textbf{(DFG)}.}
\centerline{\includegraphics[height=4cm,width=9cm]{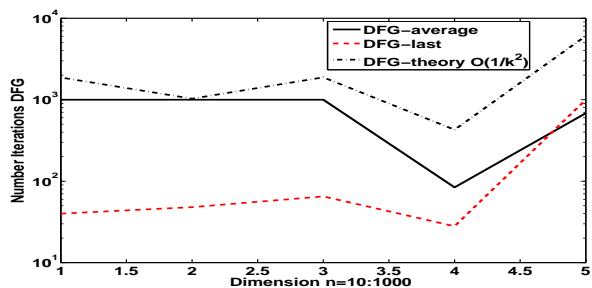}}
\end{figure}

\noindent Finally, we drop  the simple box constraints (i.e. now  $U =
\rset^n$) and for dimension $n=10^2$ we  plot in Fig. 5
the behavior of Algorithm \textbf{(DG)} in the last iterate along
iterations, starting from $x^0=0$. From our results (see Section
\ref{sec_lin}) we have linear convergence, which is also seen in
practice from this figure (in logarithmic scale). In the same figure
we also  plot the  theoretical sublinear estimates for the
convergence rate of Algorithm \textbf{(DG)} in the last iterate as
given in Section \ref{sublinear_first} (see  \eqref{fes_sublin} and
\eqref{opt_sublin}). The plot clearly confirms our theoretical
findings, i.e. linear convergence of Algorithm \textbf{(DG)} in the
last iterate, provided that $U=\rset^n$.
\begin{figure}[ht]
\label{fig_lc} \caption{Linear convergence of Algorithms
\textbf{(DG)} and \textbf{(R-DFG)} in the last iterate for $U = \rset^n$: logarithmic
scale of primal suboptimality. We also compare
with the  theoretical sublinear estimates (dot lines) for the
convergence rate in the last iterate. The plot clearly shows  our theoretical
findings, i.e. linear convergence. }
\centerline{\includegraphics[height=5.5cm,width=16cm]{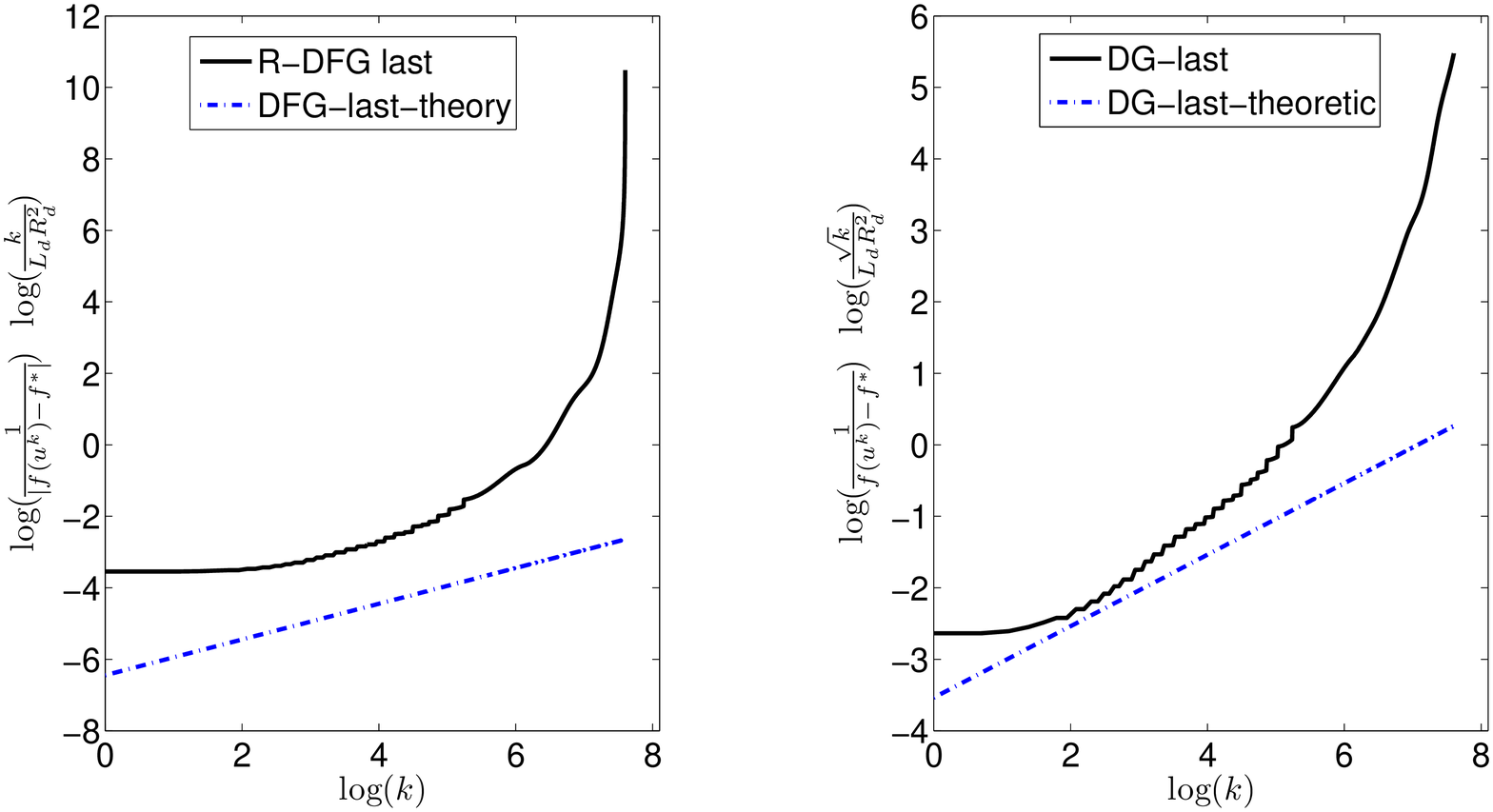}}
\end{figure}


\end{document}